\documentclass[12pt,a4]{amsart}
\usepackage{amsmath}
\usepackage{amssymb}
\usepackage{amsthm}
\usepackage{enumitem}
\usepackage{url}
\usepackage{soul}
\usepackage{times}
\usepackage[T1]{fontenc}
\usepackage[export]{adjustbox}
\usepackage{color}
\usepackage{dsfont}
\usepackage{epsfig}
\usepackage[hidelinks]{hyperref}
\usepackage{setspace}
\usepackage{epstopdf}
\usepackage[authoryear,round,sort]{natbib}
\usepackage{comment}
\usepackage{booktabs}
\usepackage{bm}
\usepackage{float}
\usepackage[linesnumbered,ruled,vlined]{algorithm2e}
\SetKwInput{KwInput}{Input}                % Set the Input
\SetKwInput{KwOutput}{Output} 
\usepackage{stmaryrd}
\usepackage{subcaption,graphicx}
\usepackage{float}
\numberwithin{equation}{section}

\usepackage[font=small,labelfont=bf]{caption}
\DeclareCaptionFont{tiny}{\tiny}
\newcommand{\norm}[1]{\left\lVert#1\right\rVert}
\usepackage{geometry}
\theoremstyle{plain}

\newtheorem{theorem}{Theorem}[section]
\newtheorem{lemma}{Lemma}[section]
\newtheorem{corollary}{Corollary}[section]

\theoremstyle{definition}
\newtheorem{definition}{Definition}[section]
\newtheorem{assumption}{Assumption}[section]

\theoremstyle{remark}
\newtheorem{rk}{Remark}[section]
\expandafter\let\expandafter\oldproof\csname\string\proof\endcsname
\let\oldendproof\endproof
\renewenvironment{proof}[1][\proofname]{%
  \oldproof[\noindent\textbf{#1.} ]%
}{\oldendproof}
\newcommand{\1}{\mathds{1}}

\newcommand{\mcH}{\mathcal{H}}
\newcommand{\miY}{X_{1,\alpha,c}^f(s)}

\newcommand{\modH}{\mathcal{H}_{1,\alpha_t,c}^f}
\newcommand{\modX}{X_{1,\bm{\alpha},c}^f}
\newcommand{\modP}{P_{1,\bm{\alpha},c}^f}
\newcommand{\modpi}{\mu_{1,\alpha_t,c}^f}
\newcommand{\modM}{M_{1,\alpha_t,c}^f}
\newcommand{\modZ}{Z_{1,\alpha_t,c}^f}

\newcommand{\modcH}{\mathcal{H}_{1,\alpha,c}^f}
\newcommand{\modcX}{X_{1,\alpha,c}^f}
\newcommand{\modcP}{P_{1,\alpha,c}^f}
\newcommand{\modcpi}{\mu_{1,\alpha,c}^f}
\newcommand{\modcM}{M_{1,\alpha,c}^f}
\newcommand{\modcZ}{Z_{1,\alpha,c}^f}

\newcommand{\mP}{\mathbb{P}}
\newcommand{\be}{\begin{equation}}
\newcommand{\ee}{\end{equation}}
\newcommand{\by}{\begin{eqnarray*}}
\newcommand{\ey}{\end{eqnarray*}}

\renewcommand{\leq}{\leqslant}
\renewcommand{\geq}{\geqslant}
\usepackage{xcolor}
\definecolor{dark-red}{rgb}{0.4,0.15,0.15}
\definecolor{dark-blue}{rgb}{0.15,0.15,0.4}
\definecolor{medium-blue}{rgb}{0,0,0.5}
\hypersetup{
    colorlinks, linkcolor={red},
    citecolor={blue}, urlcolor={blue}
}
\allowdisplaybreaks

\begin{document}
%\setstretch{1.3}
\title[Improved annealing for sampling from multimodal distributions via landscape modification]{Improved annealing for sampling from multimodal distributions via landscape modification}
\author{Michael C.H. Choi and Jing Zhang}
\address{Yale-NUS College, Singapore and School of Data Science, The Chinese University of Hong Kong, Shenzhen and Shenzhen Institute of Artificial Intelligence and Robotics for Society, Guangdong, 518172, P.R. China}
\email{michael.choi@yale-nus.edu.sg, jingzhang1@link.cuhk.edu.cn }

\date{\today}
\maketitle

\begin{abstract}
	Given a target distribution $\mu \propto e^{-\mcH}$ to sample from with Hamiltonian $\mcH$ on a finite state space, in this paper we propose and analyze new Metropolis-Hastings sampling algorithms that target an alternative distribution $\mu^f_{1,\alpha,c} \propto e^{-\mcH^{f}_{1,\alpha,c}}$, where $\modcH$ is a landscape-modified Hamiltonian which we introduce explicitly. The advantage of the Metropolis dynamics which targets $\modcpi$ is that it enjoys reduced critical height described by the threshold parameter $c$, function $f$, and a penalty parameter $\alpha \geq 0$ that controls the state-dependent effect.
	
	First, we identify a convergence-bias tradeoff phenomenon: a larger value of $\alpha$ gives faster convergence yet a larger bias from $\mu$. To mitigate this bias, in the case of fixed $\alpha$ we propose a self-normalized estimator that corrects for the bias and prove asymptotic convergence results and Chernoff-type bound of the proposed estimator.
	
	Next, we consider the case of annealing the penalty parameter $\alpha$ so that the state-dependent effect diminishes to zero in the long run. We prove strong ergodicity
	and bounds on the total variation mixing time of the resulting non-homogeneous chain subject to appropriate assumptions on the decay of $\alpha$. 
	
	Finally, we illustrate the proposed algorithms by comparing their mixing times with the original Metropolis dynamics on a variety of statistical physics models including the ferromagnetic Ising model on the hypercube or the complete graph and the $q$-state Potts model on the two-dimensional torus. In these cases, the mixing times of the classical Glauber dynamics are at least exponential in the system size as the critical height grows at least linearly with the size, while the proposed annealing algorithm, with appropriate choice of $f$, $c$ and annealing schedule on $\alpha$, mixes rapidly with at most polynomial dependence on the size. The crux of the proof harnesses on the important observation that the reduced critical height can be bounded independently of the size that consequently gives rise to rapid mixing.
	
	We emphasize that the techniques developed in this paper are not only limited to Metropolis-Hastings or Glauber dynamics, and are broadly applicable to other sampling algorithms.
	\smallskip
	
	
\end{abstract}

\tableofcontents

%\newpage

\section{Introduction}

Sampling lies at the heart of modern computational physics and Bayesian statistics. Often we would like to generate samples efficiently from a given target distribution $\mu \propto e^{-\mcH}$ with Hamiltonian $\mcH$ on a finite state space $\mathcal{X}$, without knowing the normalization constant or the partition function of $\mu$. In this setting, the classical Metropolis-Hastings (MH) algorithm offers a way to construct Markov chains that converge to $\mu$ in the long run: a proposal chain is utilized to generate a possible move, followed by either accepting or rejecting the proposal move based on a ratio of the target $\mu$. The resulting MH chain is reversible with respect to $\mu$, and various results have been developed concerning its long-time convergence. We refer readers to \cite{RR04,BGJM11,D09} for further pointers in the literature on the MH algorithm.

While the MH chain is known to mix quickly on unimodal target \cite{JS18}, in reality, the target distribution is often multimodal and high-dimensional, and as much the MH chain can possibly get stuck at a local mode that consequently leads to torpid mixing. For instance, in Section \ref{sec:apps} we investigate the ferromagnetic Ising model on the $n$-dimensional hypercube. The total variation mixing time of the classical MH algorithm in this model is at least exponential in the number of vertices. It is, therefore, crucial to develop acceleration techniques to tackle the challenges arising from sampling multimodal $\mu$ and to possibly improve the convergence of Markov chain Monte Carlo algorithms. A wide variety of tools and algorithms have been suggested to overcome this bottleneck, for example, the equi-energy sampler \cite{KZW06,AJDD08,EL20}, simulated tempering or replica exchange methods \cite{BR16,MZ03,WSH09}, and more recently non-reversible Markovian Monte Carlo algorithms \cite{Bie16,Hwang93,Hwang05}, to name but a few.

The core of this paper centers around a technique that we call \textit{landscape modification}. The high-level idea is to run algorithms on an alternative landscape with reduced critical height that hopefully accelerate the convergence of the Markov chain. The starting point of this work originates from \cite{FQG97}, in which the authors therein analyze a variant of the overdamped Langevin diffusion with state-dependent diffusion coefficient for simulated annealing. In \cite{C20KSA}, it is realized that the state-dependent effect can be effectively cast to kinetic simulated annealing by simply altering the target Hamiltonian function to optimize, and in \cite{C21} the same idea is applied in the discrete setting. 

While the focus of the previous papers \cite{C20KSA,C21} are on global optimization, in this paper we propose and analyze new MH sampling algorithms that target an alternative distribution $\mu^f_{1,\alpha,c} \propto e^{-\mcH^{f}_{1,\alpha,c}}$, where $\modcH$ is a landscape-modified Hamiltonian which we shall define explicitly in Section \ref{sec2}. One significant benefit of the Metropolis dynamics which targets $\modcpi$ is that it enjoys reduced critical height which depends on the threshold parameter $c$, function $f$ and a penalty parameter $\alpha \geq 0$ that controls the state-dependent effect. 

We summarize the major contributions of this paper as follows:

\begin{itemize}
	\item \textbf{Propose and analyze a self-normalized estimator in the case of fixed penalty}:
	To mitigate the bias of sampling from $\mu^f_{1,\alpha,c}$ instead of $\mu$, we propose a self-normalized estimator that correct for the bias and prove both asymptotic and non-asymptotic convergence results of the proposed estimator.
	
	\item \textbf{Propose and analyze an annealing algorithm on the penalty parameter}:
	The second algorithm that we propose involves annealing the penalty parameter $\alpha$. We prove strong ergodicity and bounds on the total variation mixing time of the resulting non-homogeneous Markov chain subject to appropriate assumptions on the decay of $\alpha$. 
	
	\item \textbf{Rapid mixing of the proposed annealing algorithm on the ferromagnetic Ising model on complete graph or hypercube and the $q$-state Potts model on the two-dimensional torus}:
	We illustrate the benefits of utilizing landscape modification in the ferromagnetic Ising model on the hypercube or the complete graph and the $q$-state Potts model on the two-dimensional torus. In these cases, the mixing times of the classical Glauber dynamics are at least exponential in the system size as the critical height grows at least linearly with the number of vertices, while the proposed annealing algorithm, with appropriate choice of $f$, $c$ and annealing schedule on $\alpha$, mixes rapidly with at most polynomial dependence on the system size. The crux of the proof relies on the observation that the reduced critical height can be bounded independently of the size that consequently leads to rapid mixing.
\end{itemize}

We stress that the techniques developed in this paper are not only limited to MH, and are in fact broadly applicable to other sampling algorithms by simply altering the target distribution from $\mu$ to $\mu^f_{1,\alpha,c}$, followed by either annealing the penalty parameter or to correct the bias via self-normalization.

The rest of this paper is organized as follows. In Section \ref{subsec:notations}, we fix some commonly-used notations throughout the paper. In Section \ref{sec2}, we first recall the classical MH and the notion of critical or communication height, followed by defining the landscape-modified Hamiltonian $\modH$ and the MH chain on the modified landscape. We conclude the Section by calculating explicitly $\modH$ when $f$ is a linear, quadratic or exponential function. We then proceed to bound the spectral gap with reduced critical height of the proposed chain in Section \ref{sec:spectralgapreduced}. In Section \ref{sec:constantpenalty}, we consider the case of constant penalty and introduce a self-normalized estimator. We prove related limit theorems as well as non-asymptotic Chernoff bounds of the estimator. In Section \ref{sec:timedpenalty}, we investigate the situation with time-dependent penalty and offer bounds on the total variation mixing time. Finally, we illustrate the model-independent results obtained in previous Sections in the ferromagnetic Ising model on hypercube or complete graph and the $q$-state Potts model on the two-dimensional torus in Section \ref{sec:apps}.

\subsection{Notations}\label{subsec:notations}

Throughout this paper, we adopt the following notations. For $x,y \in \mathbb{R}$, we write $x_+ = \max\{x,0\}$ to denote the non-negative part of $x$, and $x \wedge y = \min\{x,y\}$. We also write $\1_{A}$ to be the indicator function of the set $A$.

In order to compare the efficiency of algorithms, we now introduce a few common asymptotic notations. For two functions $g_1, g_2 : \mathbb{R} \to \mathbb{R}$, we say that $g_1 = \mathcal{O}(g_2)$ if there exists a constant $C > 0$ such that for sufficiently large $x$, we have $|g_1(x)| \leq C g_2(x)$. On the other hand, we use $\Omega$ to denote the tight lower bound, that is, we say that  $g_1 = \Omega(g_2)$ if there exists a constant $C > 0$ such that for sufficiently large $x$, we have $|g_1(x)| \geq C g_2(x)$. The $\ell^{\infty}$-norm of a bounded $g$ is defined as $\|g\|_{\infty}=\sup_x \{|g(x)|\}$.

%In order to compare efficiency of algorithms, we now introduce some asymptotic notations, $\mathcal{O}$ and $o$ are used as a tight  and loose upper-bound on the growth of an algorithm, respectively. That is, for two functions $g_1, g_2 : \mathbb{R} \to \mathbb{R}$, we say that $g_1 = \mathcal{O}(g_2)$ if there exists a constant $C > 0$ such that for sufficiently large $x$, we have $|g_1(x)| \leq C g_2(x)$. We also write $g_1 = o(g_2)$ if $\lim_{x \to \infty} g_1(x)/g_2(x) = 0$, and denote $g_1 \sim g_2$ if $\lim_{x \to \infty} g_1(x)/g_2(x) = 1$. We say that $g_1(x)$ is a subexponential function if $\lim_{x \to \infty} \frac{1}{x} g_1(x) = 0$.  On the other side, $\Omega$ describe the tight lower bound notation. I.e,.   $g_1 = \Omega(g_2)$ if there exists a constant $C > 0$ such that for sufficiently large $x$, we have $|g_1(x)| \geq C g_2(x)$.

We say that a sequence of random variables $(X_n)_{n\in\mathbb{N}}$ converges almost surely to the random variable $X$, denoted  as
$X_n\rightarrow^{a.s.}X$, if
$
\mathbb{P}\left(\lim_{n \rightarrow \infty}X_n=X\right) = 1
$. Meanwhile, we say that  $(X_n)_{n\in\mathbb{N}}$ converges in distribution to $X$, shown as
$X_n\rightarrow^{d}X$, if $ \lim_{n\to\infty}\mathbb{P}(X_n\leq x)=\mathbb{P}(X\leq x)$ for any $x\in\mathbb{R}$. Throughout the paper, we denote $\mP_{x}$ to be the probability measure when the Markov chain is initialized at $x$, and $\mathbb{E}_\mu(X)$ denotes the expectation of $X$ with respect to $\mu$. 
In Section \ref{sec:timedpenalty}, we introduce the time-dependent penalty parameter $(\alpha_t)_{t \geq 0}$ and write  $\alpha'_{t} = d\alpha_t/dt$ as the derivative of $\alpha_{t}$ with respect to $t$.

\section{Preliminaries}\label{sec2}

Let $\mcH: \mathcal{X} \rightarrow \mathbb{R}$ be a given Hamiltonian function on a finite state space $\mathcal{X}$. The primary objective in this paper is to sample from the target distribution $\mu \propto e^{-\mcH}$. Among various Markov chain Monte Carlo algorithms to sample from $\mu$, we shall be interested in the classical MH chain and its variants in the continuous-time setting. The MH algorithm offers a way to construct Markov chain that is reversible with respect to the target $\mu$, and now we recall its definition:
\begin{definition}[Classical MH with target distribution $\mu$]\label{def:MHpi}
	Given a target distribution $\mu$ on $\mathcal{X}$ and a proposal continuous-time Markov chain $X^0 = (X^0(t))_{t \geq 0}$ with symmetric generator $Q$, the MH chain has generator given by $M^0 = M^0(Q,\mu) = (M^0(x,y))_{x,y \in \mathcal{X}}$, where
	$$M^0(x,y) = M^0(Q,\mu)(x,y) :=\begin{cases} \min\left\{Q(x,y),\dfrac{\mu(y)}{\mu(x)}Q(y,x)\right\} = Q(x,y) e^{- (\mcH(y)- \mcH(x))_+}, &\mbox{if } x \neq y, \\ 
	- \sum_{y: x \neq y} M^0(x,y), & \mbox{if } x = y. \end{cases}$$
	We shall also denote the transition semigroup to be $((P^0)^t)_{t \geq 0}$.
\end{definition}
We remark that the superscript of $0$ that appears in Definition \ref{def:MHpi} will be explained in Section \ref{subsec:MHpif} below: it turns out that in the special case when $f = 0$, the modified MH dynamics that we propose reduces to the classical MH.

%  Let $\mcH: \mathcal{X} \rightarrow \mathbb{R}$ be a given Hamiltonian function on a finite state space $\mathcal{X}$. Our target is  to locate the set
%$$
%\bm{x}_{\min} := \left\{x \in\mathcal{X}: \mcH(x)=\min _{y \in \mathcal{X}}\mcH(y)\right\}
%$$
%of minimizing target function $\mcH$. In probabilistic terms, the algorithm consists of a  continuous Markov chain  $X  = (X_t)_{t \geq 0}$, with  infinitesimal generator,  transition matrices and invariant distribution denoted as $M = (M_{ij})_{i,j \in \mathcal{X}}$, $P = (P_{ij})_{i,j \in \mathcal{X}}$ and $(\mu_i)_{i  \in \mathcal{X}}$, respectively.  

We equip the Hilbert space $\ell^2(\mu) $ with the usual inner product weighted by the distribution $\mu $: for $g_1,g_2 \in \ell^2 (\mu )$,
$$\langle g_1,g_2 \rangle_{\mu } := \sum_{x \in \mathcal{X}} g_1(x) g_2(x) \mu (x),~~~~\|g_1\|_{\mu}=\langle g_1,g_1\rangle_\mu,~~\text{ and }~~ \mu(g_1)=\langle g_1,1\rangle_\mu.$$ 
For any infinitesimal generator $M$ of an ergodic continuous-time Markov chain $X$ with stationary distribution $\mu$, for instance $M^0$, we let $\{0=\lambda_1,\lambda_2,...,\lambda_{|\mathcal{X}|}\}$ be the eigenvalues, counted with their algebraic multiplicities and arranged in non-descending order, of $-M$. In particular, $\lambda_2(-M)$ is the spectral gap of $X$, that is,
\begin{align}\label{eq:spectralgap}
\lambda_2(-M ) := \inf_{g \in \ell^2(\mu): \mu(g) = 0} \dfrac{\langle -Mg,g \rangle_{\mu }}{\langle g,g \rangle_{\mu }}.
\end{align} The spectral gap controls the rate at which a Markov chain converges to its stationary distribution.  
In this paper, we shall be interested in bounding the total variation mixing time. To this end, let us recall that for any probability measure $\nu_1,\nu_2$ with support on $\mathcal{X}$, the \textit{total variation distance} between $\nu_1$ and $\nu_2$ is 
$$||\nu_1 - \nu_2||_{TV} := \sup_{A \subset \mathcal{X}} |\nu_1(A) - \nu_2(A)| = \dfrac{1}{2} \sum_{x \in \mathcal{X}} |\nu_1(x) - \nu_2(x)|.$$  
The total variation mixing time of $X^0$ is then defined as: 
\begin{equation}\label{eq:mixingtime}
t_{mix}^0(\varepsilon) := \inf \bigg\{t \geq 0;~ 	\max_{x \in \mathcal{X}} \left\|\left(P^0\right)^t(x,\cdot)-\mu  \right\|_{TV} \leq \varepsilon\bigg\}.
\end{equation} 
To bound the spectral gap, we recall the concept of critical height or communication height as in the metastability literature \cite{HS87,BH15}. For a given irreducible proposal chain with generator $Q$, a path from $x$ to $y$ is any sequence of points starting from $x_0 = x, x_1, x_2,\ldots, x_n =y$ such that $Q(x_{i-1},x_i) > 0$ for $i = 1,2,\ldots,n$. For any $x \neq y$, such path exists as $Q$ is irreducible. We write $\chi^{x,y}$ to be the set of paths from $x$ to $y$, and elements of $\chi^{x,y}$ are denoted by $\omega = (\omega_i)_{i=0}^n$. The highest value of the original Hamiltonian function $\mcH$ along a path $\omega \in \chi^{x,y}$, known as the elevation, is defined to be 
$$\mathrm{Elev}(\omega) := \max_{\omega_i \in \omega}\{\mcH(\omega_i)\},$$ 
and the lowest possible highest elevation along path(s) from $x$ to $y$ is 
$$H(x,y) := \min_{\omega \in \chi^{x,y}}\{\mathrm{Elev}(\omega)\}.$$
 
For $X^0$, the associated \textit{critical height} is defined to be
\begin{align}\label{eq:chland}
m^0  &:= \max_{x,y \in \mathcal{X}}\{H (x,y) - \mcH(x) - \mcH(y) + \min_x \mcH(x)\}.
\end{align}

 \subsection{Metropolis-Hastings with landscape-modified Hamiltonian}\label{subsec:MHpif}
 
 Instead of directly targeting  $\mcH$ in the Gibbs distribution $\mu (x)\propto e^{-\mcH(x)}$ in MH, we shall consider sampling from an alternative Gibbs distribution 
 \begin{align*}
 	\modcpi(x) = \dfrac{e^{-\modcH(x)} }{\modcZ} \propto e^{-\modcH(x)} 
 \end{align*} in the MH chain, where
	\begin{align}\label{eq:Zf}
		\modcZ := \sum_{x \in \mathcal{X}} e^{-\modcH(x)}
	\end{align}
	is the normalization constant. The modified or transformed Hamiltonian function $\modcH$ depends on three parameters: the function $f$, the threshold parameter $c$ and the penalty parameter $\alpha$. These parameters are assumed to be chosen in the following manner: 
 \begin{assumption}\label{assump:main}
 	\begin{enumerate}		
 		\item\label{it:assumpf} The function $f: \mathbb{R}^+ \to \mathbb{R}^+$ is differentiable, non-negative and non-decreasing. Furthermore, $f$ satisfies
 		$$f(0) = 0.$$
 		
 		\item\label{it:assumpc} The threshold parameter $c$ is chosen so that $\mcH_{\max} \geq c \geq \mcH_{\min}$, where $\mcH_{\max} := \max_{x \in \mathcal{X}} \mcH(x)$ and $\mcH_{\min} := \min_{x \in \mathcal{X}} \mcH(x)$ are respectively the global maximum and minimum value of $\mcH$.
 		
 		\item\label{it:assumpalpha} The penalty parameter $\alpha$ is non-negative, i.e. $\alpha \geq 0$.
 	\end{enumerate}
 \end{assumption}

In the rest of this paper, we suppose that Assumption \ref{assump:main} holds. We now state the definition of $\modcH$:
\begin{definition}[Landscape-modified Hamiltonian $\mcH_{\epsilon,\alpha,c}^f$]
	Suppose Assumption \ref{assump:main} holds. At temperature $\epsilon > 0$, the landscape-modified function $\mcH_{\epsilon,\alpha,c}^f$ is
	\begin{align*}
		\mcH^f_{\epsilon,\alpha,c}(x) &:= \int_{\mcH_{\min}}^{\mcH(x)} \dfrac{1}{\alpha f((u-c)_+) + \epsilon}\,du.
	\end{align*}
	In particular, if we take $\epsilon = 1$, then
	\begin{align*}
		\modcH(x) = \int_{\mcH_{\min}}^{\mcH(x)} \dfrac{1}{\alpha f((u-c)_+) + 1}\,du =  \begin{cases}
			\mcH(x) - \mcH_{\min},\quad & \text{if} \, x \in \{\mcH(x) < c\}, \\
			c - \mcH_{\min} + \int_{c}^{\mcH(x)} \dfrac{1}{\alpha f(u-c) + 1}\,du, \quad & \text{if} \, x \in \{\mcH(x) \geq c\}. \\
		\end{cases}
	\end{align*}
	If we further specialize into $f = 0$ or $\alpha = 0$, then we retrieve the original target distribution as $\mu = \mu^0_{1,\alpha,c} = \mu^f_{1,0,c}$.
\end{definition}

We shall see in Section \ref{sec:spectralgapreduced} that one important benefit of running MH on the modified landscape $\modcH$ is that its associated critical height is less than or equal to the original critical height $m^0$.

To understand $\modcH$ intuitively, we recall that this transformation is first introduced in \cite{C21} for simulated annealing in the discrete setting for global optimization, and is inspired by the overdamped Langevin diffusion with state-dependent diffusion coefficient in \cite{FQG97}. The parameter $c$ can be interpreted as a threshold parameter: on $\{\mcH(x) < c\}$, $\modcH$ shares the same scale as the original $\mcH$; while on $\{\mcH(x) \geq c\}$, the state-dependent effect is activated since the integral
$$\int_{c}^{\mcH(x)} \dfrac{1}{\alpha f(u-c) + 1}\,du \leq \mcH(x) - c.$$
Thus, the function $f$ can be understood as determining the magnitude of the state-dependent effect: if we choose a larger $f$, the above integral on the left hand side and hence $\modcH$ are smaller. On the other hand, this effect is also controlled and penalized by $\alpha$. A larger $\alpha$ gives smaller $\modcH$ yet a larger bias between $\modcpi$ and $\mu$. The name penalty parameter is largely motivated by the work in \cite{KD20}, where the authors therein consider a penalized Langevin diffusion in which the negative of the gradient of the potential function is perturbed with a penalty that vanishes as time increases.

% with the following modified or transformed function $\mcH_{\varepsilon,\alpha_{t},c}^f$ at temperature $\varepsilon = 1$:
% \begin{align}\label{eq:mcHeps}
% \notag\mcH_{\varepsilon}(x) &= \mcH^f_{\varepsilon,\alpha_t,c}(x) := \int_{\mcH_{\min}}^{\mcH(x)} \dfrac{1}{\alpha_{t} f((u-c)_+) + \varepsilon}\,du,\\
% &= \modH(x) := \int_{\mcH_{\min}}^{\mcH(x)} \dfrac{1}{\alpha_{t} f((u-c)_+) + 1}\,du,
% \end{align}
% where $\mcH_{\min} := \min_{x \in \mathcal{X}} \mcH(x)$ denotes the global minimum value of $\mcH$ and the function $f$ and the parameter $c$ are chosen to satisfy the following assumptions:
 
 We now define formally the MH chain that targets $\modcpi$:
 
 \begin{definition}[MH with target distribution $\modcpi$]\label{def:MHpif}
 	Suppose Assumption \ref{assump:main} holds. Given the target distribution $\modcpi$ on $\mathcal{X}$ and a proposal continuous-time Markov chain $\modcX= (\modcX(t))_{t \geq 0}$ with symmetric generator $Q$, the MH chain has generator given by $\modcM= \modcM(Q,\mu) = (\modcM(x,y))_{x,y \in \mathcal{X}}$, where
 	$$\modcM(x,y) :=\begin{cases} \min\left\{Q(x,y),\dfrac{\modcpi(y)}{\modcpi(x)}Q(y,x)\right\} = Q(x,y) e^{- (\modcH(y)-\modcH(x))_+}, &\mbox{if } x \neq y, \\ 
 		- \sum_{y: x \neq y} \modcM(x,y), & \mbox{if } x = y. \end{cases}$$
 	We shall also denote the transition semigroup to be $((\modcP)^t)_{t \geq 0}$, and its spectral gap to be $$\lambda_2(-\modcM) := \inf_{g \in \ell^2(\modcpi): \modcpi(g) = 0} \dfrac{\langle -\modcM g,g \rangle_{\modcpi }}{\langle g,g \rangle_{\modcpi}}.$$ Note that if $f = 0$ or $\alpha = 0$, the above dynamics reduce to the classical MH with target distribution $\mu$ since
 	$$M^0 = M^0_{1,\alpha,c} = M^f_{1,0,c}.$$
 \end{definition}
 
 \begin{rk}
 	In view of Definition \ref{def:MHpif}, the knowledge of $\mcH_{\min}$ is not required in simulating $\modcX$ since the difference, for $x,y \in \mathcal{X}$,
 	$$\modcH(y)- \modcH(x) = \int_{\mcH(x)}^{\mcH(y)} \dfrac{1}{\alpha f((u-c)_+) + 1}\,du$$
 	does not depend on $\mcH_{\min}$.
 \end{rk}
% The chain with new modification $\modX=\left(\modX(t)\right)_{t \geqslant 0}$ has target distribution $ \modpi(x) \propto e^{-\modH(x)} $. Denote its   transition matrices as  $\modP=(\modP(x,y))_{x,y\in\mathcal{X}}$, infinitesimal generator as $\modM=(\modM(x,y))_{x,y\in\mathcal{X}}$ and corresponding spectral gap as $\lambda_2(-\modM)$. 
%  Note that when we take $\alpha_{t} = 0,$ the above dynamics also reduce to $X$.
  
In the following subsections, we compute explicitly the landscape-modified Hamiltonian $\modcH$ with linear, quadratic or exponential $f$.

\subsection{Landscape-modified Hamiltonian $\modcH$ with $f(x) = x$}\label{subsec:linearf}

In this subsection, we compute explicitly $\modcH$ when we specialize into $f(x) = x$, which gives
\begin{align*}
	\modcH(x) = \begin{cases}
	\mcH(x) - \mcH_{\min},\quad & \text{if} \, x \in \{\mcH(x) < c\}, \\
	c - \mcH_{\min} + \dfrac{1}{\alpha} \ln (1 + \alpha (\mcH(x) - c)), \quad & \text{if} \, x \in \{\mcH(x) \geq c\}. \\
	\end{cases}
\end{align*}

Using the fact that
$$\lim_{\alpha \to 0} \dfrac{1}{\alpha} \ln (1 + \alpha (\mcH(x) - c)) = \mcH(x) - c,$$
we thus obtain, for $x \in \mathcal{X}$,
$$\lim_{\alpha \to 0} \modcH(x) = \mcH(x)-\mcH_{\min}.$$
In other words, we retrieve the original target distribution $\mu\propto  {e^{-\mcH}}\propto e^{-(\mcH-\mcH_{\min})
}$ when we take $\alpha \to 0$, as expected.

\subsection{Landscape-modified Hamiltonian $\modcH$ with $f(x) = x^2$}\label{subsec:quadraticf}

In this subsection, we write down explicitly $\modcH$ when we specialize into $f(x) = x^2$, which yields
\begin{align}\label{eq:Hal_x2}
\modcH(x) = \begin{cases}
\mcH(x) - \mcH_{\min},\quad & \text{if} \, x \in \{\mcH(x) < c\}, \\
c - \mcH_{\min} + \dfrac{1}{\sqrt{\alpha}} \arctan( \sqrt{\alpha} (\mcH(x) - c) ), \quad & \text{if} \, x \in \{\mcH(x) \geq c\}. \\
\end{cases}
\end{align}

As a result, we observe that 
$$\max_x \modcH(x) \leq c - \mcH_{\min} + \dfrac{1}{\sqrt{\alpha}} \dfrac{\pi}{2}.$$
Consequently, if we take $c = \mcH_{\min} + \delta$, where $\delta > 0$, the maximum of $\modcH$ can be bounded above independently of $\mcH_{\max}$ or $\mcH_{\min}$. This important observation will be utilized in Section \ref{sec:apps} on a few statistical physics models to illustrate rapid mixing of the proposed annealing algorithm in Section \ref{sec:timedpenalty}. 

\subsection{Landscape-modified Hamiltonian $\modcH$ with $f(x) = e^x-1$}\label{subsec:expcf}

In this subsection, we calculate explicitly $\modcH$ when we specialize into $f(x) = e^x-1$, which yields
\begin{align}\label{eq:Hal_exp}
\modcH(x) = \begin{cases}
\mcH(x) - \mcH_{\min},\quad & \text{if} \, x \in \{\mcH(x) < c\}, \\
c - \mcH_{\min} +\dfrac{\log \left(\alpha\left(e^{\mcH(x)}-1\right)+1\right)-\log \left(\alpha\left(e^{c}-1\right)+1\right)-(\mcH(x)-c)}{\alpha-1} , \quad & \text{if} \, x \in \{\mcH(x) \geq c\}. \\
\end{cases}
\end{align}
Also, we know that 
$$\lim_{\alpha \to 0} \dfrac{\log (\alpha(e^{\mcH(x)}-1)+1)}{\alpha-1}   = 0,$$
thus 
$$\lim_{\alpha \to 0} \dfrac{\log \left(\alpha\left(e^{\mcH(x)}-1\right)+1\right)-\log \left(\alpha\left(e^{c}-1\right)+1\right)-(\mcH(x)-c)}{\alpha-1}=H(x)-c,$$ 
which implies that, for $x \in \mathcal{X}$,
$$\lim_{\alpha \to 0} \modcH(x) = \mcH(x)-\mcH_{\min}.$$
In other words, we recover the original target distribution $\mu\propto  {e^{-\mcH}}\propto e^{-(\mcH-\mcH_{\min})
}$ when we take $\alpha \to 0$, as expected.

\section{Spectral gap bounds in terms of the reduced critical height}\label{sec:spectralgapreduced}

In this Section, we bound the spectral gap of $X^f_{1,\alpha,c}$, the MH  dynamics with target distribution $\pi^f_{1,\alpha,c}$ as introduced in Definition \ref{def:MHpif}, in terms of its associated critical height. It turns out that this critical height is reduced and is less than or equal to the original critical height $m^0$, which demonstrates one advantage of running MH algorithms on the modified landscape described by $\mcH^f_{1,\alpha,c}$.

Recall that we write $\chi^{x,y}$ to be the set of paths from $x$ to $y$ as specified by the proposal generator. The highest value of the landscape-modified Hamiltonian function $\mcH^f_{1,\alpha,c}$ along a path $\omega \in \chi^{x,y}$ is defined to be 
$$\mathrm{Elev}^f_{1,\alpha,c}(\omega) := \max_{\omega_i \in \omega}\{\mcH^f_{1,\alpha,c}(\omega_i)\},$$ 
and the lowest possible highest elevation along path(s) from $x$ to $y$ is 
$$H^f_{1,\alpha,c}(x,y) := \min_{\omega \in \chi^{x,y}}\{\mathrm{Elev}^f_{1,\alpha,c}(\omega)\}.$$

With the above notations in mind, we define the concept of reduced critical height:

\begin{definition}[Reduced critical height]
	\label{def:Reduced critical height}
	Given the same proposal chain with generator $Q$ of both $X^0$ and $X^f_{1,\alpha,c}$, Then the critical height of $X^f_{1,\alpha,c}$ is defined to be
	$$m_{1,\alpha, c}^{f} := \max_{x,y \in \mathcal{X}}\{H^f_{1,\alpha,c} (x,y) - \mcH^f_{1,\alpha,c}(x) - \mcH^f_{1,\alpha,c}(y) \} . $$
\end{definition}

\begin{rk}\label{rk:reducedchbound}
	Suppose $(x^*,y^*)$ attains the critical height $m^f_{1,\alpha,c}$ of $X^f_{1,\alpha,c}$ in the sense that it satisfies $m^f_{1,\alpha,c} =\mcH^f_{1,\alpha,c} (x^*)-\mcH^f_{1,\alpha,c}(y^*)$.
	Since $\alpha \geq 0$ and $f\geq 0$ we thus have 
%	$$\int_{\mcH\left(y^{*}\right)}^{\mcH\left(x^{*}\right)} \frac{1}{\alpha f\left((u-c)_{+}\right)+1} \,du = \mcH^f_{1,\alpha,c}(x^*) - \mcH^f_{1,\alpha,c}(y^*)$$
	\begin{align*}
	m_{1,\alpha, c}^{f}  
	=\int_{\mcH\left(y^{*}\right)}^{\mcH\left(x^{*}\right)} \frac{1}{\alpha f\left((u-c)_{+}\right)+1} \, d u 
	\leq \mcH(x^*)-\mcH(y^*) \leq m^0.
	\end{align*}
	Thus, the reduced critical height $m_{1,\alpha, c}^{f}$ is bounded above by the original critical height $m^0$. This illustrates one advantage of using the chain $X^f_{1,\alpha,c}$.
\end{rk}

%\begin{theorem}\label{thm:chland2}
%	$$m^f_{1,\alpha,c}  = \max_{x,y \in \mathcal{X}}\{G^f_{1,\alpha,c} (x,y) - \mcH^f_{1,\alpha,c}(x) - \mcH^f_{1,\alpha,c}(y) \}.$$
%\end{theorem} 
%
%\begin{proof}
%	First, we note that $\mcH(x) \leq \mcH(y)$ holds if and only if $\mcH^f_{1,\alpha,c}(x) \leq \mcH^f_{1,\alpha,c}(y)$. In other words, the monotonicity structure is preserved under landscape modification.
%	
%	We first prove that 
%	$$m^f_{1,\alpha,c}  \leq \max_{x,y \in \mathcal{X}}\{G^f_{1,\alpha,c} (x,y) - \mcH^f_{1,\alpha,c}(x) - \mcH^f_{1,\alpha,c}(y) \}$$
%	Suppose we have
%	$$m^0 = \mcH(x^*) - \mcH(y^*) = G(z^*,y^*) - \mcH(y^*),$$
%	where $z^*$ is a global minimum of $\mcH$. Then 
%	$$\max_{x,y \in \mathcal{X}}\{G^f_{1,\alpha,c} (x,y) - \mcH^f_{1,\alpha,c}(x) - \mcH^f_{1,\alpha,c}(y) \} \geq G^f_{1,\alpha,c}(z^*,y^*) - \mcH^f_{1,\alpha,c}(y^*) = \mcH^f_{1,\alpha,c}(x^*) - \mcH^f_{1,\alpha,c}(y^*) = m^f_{1,\alpha,c}.$$
%	
%	Assume the contrary that 
%	$$\max_{x,y \in \mathcal{X}}\{G^f_{1,\alpha,c} (x,y) - \mcH^f_{1,\alpha,c}(x) - \mcH^f_{1,\alpha,c}(y) \} =G^f_{1,\alpha,c}(z_2^*,y_2^*) - \mcH^f_{1,\alpha,c}(y_2^*) = \mcH^f_{1,\alpha,c}(x_2^*) - \mcH^f_{1,\alpha,c}(y_2^*)> m^f_{1,\alpha,c}.$$
%	
%	We then consider two cases:
%	\begin{itemize}
%		\item Case 1: $\mcH(x_2^*) - \mcH(y_2^*) > \mcH(x^*) - \mcH(y^*)$\newline
%		This is impossible since this case violates the definition of $m^0$.
%		
%		\item Case 2: $\mcH(x_2^*) - \mcH(y_2^*) \leq \mcH(x^*) - \mcH(y^*)$\newline
%		
%	\end{itemize}
%\end{proof}

Using \cite[Lemma $2.7$]{HS88} with the Hamiltonian function therein be replaced by $\mcH^f_{1,\alpha,c}$ and at temperature $1$, we can readily bound the spectral gap $\lambda_2(-M^f_{1,\alpha,c})$ in terms of $m^f_{1,\alpha,c}$: 
\begin{lemma}\label{lem:spectralgap}
	\[
	\tilde{C}e^{-m_{1,\alpha,c}^f}\geq \lambda_2(-\modM)\geq Ce^{-m_{1,\alpha,c}^f} \geq C e^{-m^0},
	\]where $C$ and $\tilde{C}$ only depend on $\mathcal{X}$ and
	$ 
	m_{1,\alpha, c}^{f}
	$ is the reduced critical height as introduced in Definition \ref{def:Reduced critical height}.
\end{lemma}

Consequently, thanks to Lemma \ref{lem:spectralgap} we obtain

\begin{corollary}\label{cor:spectralgap}
	Suppose $\alpha = (\alpha_t)_{t \geq 0}$ is time-dependent. We then have 
	\begin{equation} 
	\lambda_2(-\modM)\geq Ce^{-m^0}
	\end{equation} where $Ce^{-m^0}$ is independent of $t$, which implies that \[\int_0^\infty \lambda_2(-\modM)dt=\infty.\]
\end{corollary}
 
We remark that Corollary \ref{cor:spectralgap} is useful  in establishing strong ergodicity in Section \ref{sec:timedpenalty} when we consider the time-dependent penalty.

\section{Constant penalty $\alpha_t = \alpha$ for all $t \geq 0$}\label{sec:constantpenalty}
  In this Section, we focus in the constant penalty case, i.e. the penalty parameter takes on a fixed value with $\alpha_t = \alpha > 0$ for all $t \geq 0$. In this setting, we first identify a phenomenon that we call the (total-variation) convergence-bias tradeoff: using triangle inequality, we note that, for any $x \in \mathcal{X}$,
  \begin{align}\label{eq:convergencebias}
  	\left\|\left(\modcP\right)^t(x,\cdot)-\mu  \right\|_{TV} \leq \underbrace{\left\|\left(\modcP\right)^t(x,\cdot)-\modcpi\right\|_{TV}}_{\text{convergence}} + \underbrace{\|\modcpi-\mu \|_{TV}}_{\text{bias}}.
  \end{align}
   On one hand, we would like to choose $\alpha$ to be large so that the reduced critical height $m^f_{1,\alpha,c}$ is small, leading to fast convergence of $\left(\modcP\right)^t(x,\cdot)$ towards $\modcpi$. On the other hand, we would like to choose $\alpha$ to be small so that the bias between $\modcpi$ and $\mu$ is small when measured in terms of the total variation distance. We quantify the above statements into the following theorem:
  
  \begin{theorem}\label{thm:ConvergenceTradeoff}
  		Suppose that Assumption \ref{assump:main} holds.
  		
  	   \begin{enumerate}
  		\item\label{it:largealpha} (Large $\alpha$ leads to fast convergence): For $x \in \mathcal{X}$,
  		
  		\begin{align*}
  			\dfrac{1}{2} e^{- \lambda_2(-\modcM)t} \leq \left\|\left(\modcP\right)^t(x,\cdot) -\modcpi\right\|_{TV} \leq \dfrac{e^{\mcH_{\max}-\mcH_{\min}} |\mathcal{X}|}{2} e^{- \lambda_2(-\modcM)t}.
  		\end{align*}
  		
  		\item\label{it:smallalpha} (Small $\alpha$ gives small total variation bias): $$\|\modcpi-\mu \|_{TV} = \mathcal{O}(\alpha).$$
  	\end{enumerate} 
  \end{theorem}

The proof of Theorem \ref{thm:ConvergenceTradeoff} is given in Section \ref{subsec:Proof_Convergencetradeoff}.

%{A self-normalized estimator}

To mitigate the bias induced by running the Markov chain $X^f_{1,\alpha,c}$ to approximately sample from $\modcpi$ instead of $\mu$, in this Section, we consider a self-normalized estimator to correct the bias. Define \textit{the weight function} to be
$$w(x) = w^f_{1,\alpha,c}(x) := \dfrac{e^{-(\mcH(x) - \mcH_{\min})}}{e^{-\modcH(x)}},$$
and for any $g\in\ell^2(\mu)$, \textit{the self-normalized estimator} by 
$$\dfrac{\int_0^t w(\miY) g(\miY) \, ds}{\int_0^t w(\miY) \, ds}.$$
We remark that the weight function $w(x)$ can be computed readily since
\begin{align*}
	w(x) &= \dfrac{e^{-(\mcH(x) - \mcH_{\min})}}{e^{-\modcH(x)}} \\
	&= \dfrac{e^{-(\mcH(x) - \mcH_{\min})}}{\exp\bigg\{-(c \wedge \mcH(x) - \mcH_{\min}) - \int_{c \wedge \mcH(x)}^{\mcH(x)} \dfrac{1}{\alpha f((u-c)_+) + 1} \, du\bigg\}} \\
	&= \dfrac{e^{-\mcH(x)}}{\exp\bigg\{-c \wedge \mcH(x)  - \int_{c \wedge \mcH(x)}^{\mcH(x)} \dfrac{1}{\alpha f((u-c)_+)  + 1} \, du\bigg\}}
\end{align*}
does not depend on $\mcH_{\min}$. In the special case where $f = 0$ or $\alpha = 0$, we note that $w = 1$. Recall that $\modcZ$ is the normalization constant of $\modcpi$ \eqref{eq:Zf}, and we adopt the notation $$Z^0 := Z^0_{1,\alpha,c} = \sum_{x \in \mathcal{X}} e^{-(\mcH(x) - \mcH_{\min})}$$ to be the normalization constant when $f = 0$. We then have the following asymptotic convergence results for the proposed self-normalized estimator:

\begin{theorem}\label{thm:CLTselfnormalized}
	For $g \in \ell^2(\mu)$, as $t \to \infty$ we have
	\begin{enumerate}
		\item\label{it:strongConsistencyselfNomalized} (Strong consistency   of the self-normalized estimator)
		$$\dfrac{\int_0^t w(\miY) g(\miY) \, ds}{\int_0^t w(\miY) \, ds} \rightarrow^{a.s.} \mu(g).$$
		\item \label{it:CLTselfNomalized}(Central limit theorem of the self-normalized estimator)
		$$\sqrt{t} \left(\dfrac{\int_0^t w(\miY) g(\miY) \, ds}{\int_0^t w(\miY) \, ds} - \mu(g)\right) \rightarrow^{d} \dfrac{\modcZ}{Z^0} N(0,\mathrm{Var}_{\modcpi}(wg)).$$ 
	\end{enumerate}	
where the limiting variance is defined to be  $$\mathrm{Var}_{\modcpi}(wg) :=\lim_{t\to\infty}\mathbb{E}_{\mu}\left(  \dfrac{\int_0^t w(\miY) g(\miY) \, ds}{\int_0^t w(\miY) \, ds} - \mu(g) \right)^2.$$

	In particular, in the special case when $f = 0$ or $\alpha = 0$ so that $w = 1$, the above results reduce to the classical Markov chain law of large number and central limit theorem for the MH chain respectively.
\end{theorem} 

The proof of Theorem \ref{thm:CLTselfnormalized} is delayed to Section \ref{subsec:Proof_CLTselfnormalized}.

%\textcolor{red}{
%By Chapter 2 in \cite{KLO12},
%\begin{align*}
%\mathrm{Var}_{\modcpi}(wg) &= 2\langle wg, (-\modcM)^{-1} wg \rangle_{\modcpi} \\
%\left(\dfrac{\modcZ}{Z^0}\right)^2 \mathrm{Var}_{\modcpi\modcpi}(wg) &= 2 \langle \dfrac{\mu}{\modcpi}g, (-\modcM)^{-1} \dfrac{\mu}{\modcpi}g \rangle_{\modcpi} \\ 
%&=2 \langle g, (-\modcM)^{-1} \dfrac{\mu}{\modcpi}g \rangle_{\mu} \\
%&\leq2 \norm{g}^2 \norm{(-\modcM)^{-1}} \norm{\dfrac{\mu}{\modcpi}}.
%\end{align*}}

Our next result concerns the finite-time non-asymptotic deviation of the self-normalized estimator from the target mean $\mu(g)$. To this end, for any $\eta,\varepsilon>0$, $g\in\ell^2(\mu)$,  we define the first time such that the deviation probability is less than or equal to $\varepsilon$ to be:
\begin{align}\label{eq:conctime}
	t_{conc}(\eta, \varepsilon) := \inf \bigg\{t \geq 0;~ 	\sup_{g \in \ell^2(\mu);~\norm{g}_{\infty} \leq a} \sup_{x \in \mathcal{X}}\mP_x\left( \dfrac{\int_0^t w(\miY) g(\miY) \, ds}{\int_0^t w(\miY) \, ds} -  \mu(g) \geq \eta \right) \leq \varepsilon\bigg\}.
\end{align}

\begin{theorem}[Chernoff's bound of the self-normalized estimator]\label{thm:concentrationChernoff}
	Given $\eta, \varepsilon > 0$, $g \in \ell^2(\mu)$ such that $\norm{g}_{\infty} \leq a$, if we have
	$$t \geq \dfrac{\modcZ}{Z^0}\dfrac{32 a^2(1+\eta)^2}{\lambda_2(-\modcM)\eta^2} \ln \left(\dfrac{1}{\varepsilon \min_{x \in \mathcal{X}} \modcpi(x)}\right),$$
	then for any $x \in \mathcal{X}$,
	$$\mP_x\left( \dfrac{\int_0^t w(\miY) g(\miY) \, ds}{\int_0^t w(\miY) \, ds} -  \mu(g) \geq \eta \right) \leq \varepsilon.$$
	In particular,
	\begin{align*}
		t_{conc}(\eta, \varepsilon) = \mathcal{O}\left(|\mathcal{X}|\dfrac{ a^2(1+\eta)^2}{\lambda_2(-\modcM)\eta^2} \ln \left(\dfrac{|\mathcal{X}| \exp\{\max_{x \in \mathcal{X}}\modcH(x)\}}{\varepsilon}\right)\right).
	\end{align*}
	
\end{theorem}

The proof of Theorem \ref{thm:concentrationChernoff} can be found in Section \ref{subsec:concentrationChernoff}, which utilizes the results in \cite{Lezaud01}.

\subsection{Proof of Theorem \ref{thm:ConvergenceTradeoff}}
\label{subsec:Proof_Convergencetradeoff}
	First, we prove item \eqref{it:largealpha}. For the lower bound, let $\phi$ be an eigenfunction of $-\modcM$  with eigenvalue $\lambda \neq 0$. For $x \in \mathcal{X}$, we have 
	\begin{align*}
	e^{-\lambda t}|\phi(x)|  =\left|\left(\modcP\right)^t \phi(x)\right| &=\left|\sum_{y \in \mathcal{X}}\left[\left(\modcP\right)^t(x, y) \phi(y)-\modcpi(y) \phi(y)\right]\right|  \\&\leq\|\phi\|_{\infty} 2\left\|\left(\modcP\right)^t(x, \cdot)-\modcpi\right\|_{TV}.
	\end{align*}  
	Taking $x$ that attains $|\phi(x)|=\|\phi\|_{\infty}$ and $\lambda = \lambda_2(-\modcM)$ yield the desired result since
	$$  e^{- \lambda_2(-\modcM)t} \leq 2 \left\|\left(\modcP\right)^t(x,\cdot) -\modcpi\right\|_{TV}  .$$
	
	As for the upper bound, it follows from \cite[equation $(20.21)$]{LPW09} that 
	\begin{align*}
	2\left\|\left(\modcP\right)^t(x, \cdot)-\modcpi\right\|_{TV} &\leq e^{- \lambda_2(-\modcM)t} \sum_{y \in \mathcal{X}} \frac{\modcpi(y)}{\sqrt{\modcpi(y)\modcpi(x)}} \\&\leq \frac{e^{- \lambda_2(-\modcM)t}}{\min_{x\in\mathcal{X}}\modcpi(x) }
	\\& =\frac{\modcZ}{e^{-\max_{x\in\mathcal{X}}\modcH (x)}}e^{- \lambda_2(-\modcM)t}
	\\& \leq e^{\mcH_{\max}-\mcH_{\min}} \modcZ e^{-\lambda_2(-\modcM)t}
	\\&\leq  e^{\mcH_{\max}-\mcH_{\min}} |\mathcal{X}| e^{-\lambda_2(-\modcM)t}.
	\end{align*}
	
	Next, we prove item \eqref{it:smallalpha}. First, we compute the difference of $\modcZ$ and $Z^0$:
	\begin{align}
	\modcZ -Z^0 &=\sum_{x\in \mathcal{X}}e^{-\modcH(x)}-e^{-(\mcH(x)-\mcH_{\min})}\notag\\
	&=\sum_{x:\mcH(x)> c }e^{-\modcH(x)}-e^{-(\mcH(x)-\mcH_{\min})}\notag\\
	&\leq \sum_{x:\mcH(x)> c }\left|e^{-\modcH(x)}-e^{-(\mcH(x)-\mcH_{\min})}\right|\notag\\
	&\leq \sum_{x:\mcH(x)> c }\left| { \modcH(x)}- { (\mcH(x)-\mcH_{\min})}\right|\notag\\
	&= \sum_{x:\mcH(x)> c }\left|\int_{\mcH_{\min}}^{\mcH(x)}\left( \frac{1}{\alpha f\left((u-c)_{+}\right)+1}-1 \right)du\right|\notag\\
	&\leq  \sum_{x:\mcH(x)> c } \int_{\mcH_{\min}}^{\mcH(x)}   {\alpha f\left((u-c)_{+}\right)}    du \notag \leq\alpha {\xi } 
	\end{align}
	where $\xi := |{x:\mcH(x)> c }|  ({\mcH_{\max}-\mcH_{\min}})    {f((\mcH_{\max}-c)_{+})} $. Using the bound above, we calculate the total variation distance between $\modcpi$ and $\mu$:
	\begin{align*}
	\|\modcpi-\mu\|_{TV}&=\frac12 \sum_{x\in\mathcal{X}}|\modcpi(x)-\mu(x)|\\
	&=\frac12  \sum_{x\in\mathcal{X}}\left|\frac{e^{-\modcH(x)}}{\modcZ}-\frac{e^{-(\mcH(x)-\mcH_{\min})}}{Z^0 } \right|\\
	&=\frac12  \sum_{x\in\mathcal{X}}\left|\frac{e^{-\modcH(x)}}{\modcZ}-\frac{e^{-(\mcH(x)-\mcH_{\min})}}{\modcZ}+\frac{e^{-(\mcH(x)-\mcH_{\min})}}{\modcZ }-\frac{e^{-(\mcH(x)-\mcH_{\min})}}{Z^0} \right|\\
	&\leq \frac12 \frac{\modcZ - Z^0}{\modcZ} + \frac12 \frac{\modcZ - Z^0}{\modcZ} \leq \dfrac{\alpha \xi}{|\mathcal{X}|} e^{\max_x\modcH(x)} \leq \dfrac{\alpha \xi}{|\mathcal{X}|} e^{\mcH_{\max}-\mcH_{\min}} = \mathcal{O}(\alpha).
	\end{align*}
	
%	\textcolor{red}{
%		\begin{align*}
%		\|\modcpi-\mu\|_{TV}&=\frac12 \sum_{x\in\mathcal{X}}|\modcpi(x)-\mu(x)|\\
%		&=\frac12  \sum_{x\in\mathcal{X}}\left|\frac{e^{-\modcH(x)}}{\modcZ}-\frac{e^{-(\mcH(x)-\mcH_{\min})}}{Z } \right|\\
%		&=\frac12  \sum_{x\in\mathcal{X}}\left|\frac{e^{-\modcH(x)}}{\modcZ}-\frac{e^{-\modcH(x)}}{Z }+\frac{e^{-\modcH(x)}}{Z }-\frac{e^{-(\mcH(x)-\mcH_{\min})}}{Z} \right|\\
%		&\leq \frac12 \sum_{x\in\mathcal{X}}e^{-\modcH(x)}\left|\frac{1}{\modcZ}-\frac{1}{Z} \right|+\frac{1}{2Z } \sum_{x\in\mathcal{X}}\left|e^{-\modcH(x)}-e^{-(\mcH(x)-\mcH_{\min})} \right|\\
%		&\leq \frac12 \sum_{x\in\mathcal{X}}1\cdot \frac{1}{|\bm{x}_{\min}|^2}\alpha\xi +\frac{1}{2|\bm{x}_{\min}|}\alpha\xi \\
%		&= \frac12 |\mathcal{X}| \frac{1}{|\bm{x}_{\min}|^2}\alpha\xi +\frac{1}{2|\bm{x}_{\min}|}\alpha\xi =\mathcal{O}(\alpha).
%		\end{align*}} 

\subsection{Proof of Theorem \ref{thm:CLTselfnormalized}}
\label{subsec:Proof_CLTselfnormalized} 
First, we prove item (\ref{it:strongConsistencyselfNomalized}). Using the Markov chain law of large number, we have
	$$\dfrac{\modcZ}{Z^0} \dfrac{1}{t} \int_0^t w(\miY) g(\miY)\,ds \rightarrow^{a.s.} \dfrac{\modcZ}{Z^0} \modcpi(wg) = \mu(g).$$ 
	Applying the above result to $g = 1$ yields
	$$\dfrac{\modcZ}{Z^0} \int_0^t w(\miY) \, ds \rightarrow^{a.s.} \dfrac{\modcZ}{Z^0} \modcpi(w) = \mu(1) = 1.$$
	Collecting the two results above together with the continuous mapping theorem gives the desired result.
	%	where normalizers \[
	%	\modcZ=\sum_{i=0}^{n-1} e^{-\mcH_{1,\alpha,c,i}}, \quad\text{  and }\quad Z^0=\sum_{i=0}^{n-1} e^{-\mcH_{i}}.
	%	\]
	
Next, we prove item (\ref{it:CLTselfNomalized}).
	Using the Markov chain central limit theorem, we have
	$$\sqrt{t} \left(\frac1t \int_0^t w(\miY) g(\miY)\,ds - \modcpi(wg)\right) \rightarrow^{d} N(0,\mathrm{Var}_{\modcpi}(wg)).$$
	By Slutsky's theorem, as $\frac1t\int_0^t w(\miY)\, ds \rightarrow^{a.s.} \dfrac{Z^0}{\modcZ}$, the left-hand-side can be expressed as
	$$\sqrt{t} \left(\dfrac{\int_0^t w(\miY) g(\miY) \, ds}{\int_0^t w(\miY) \, ds}- \dfrac{\modcZ}{Z^0} \modcpi(wg)\right) \rightarrow^{d} \dfrac{\modcZ}{Z^0} N(0,\mathrm{Var}_{\modcpi}(wg)).$$
	Finally, using $\dfrac{\modcZ}{Z^0} \modcpi(wg) = \mu(g)$ gives the desired result since
	$$\sqrt{t} \left( \dfrac{\int_0^t w(\miY) g(\miY) \,ds}{\int_0^t w(\miY) \, ds}- \mu(g)\right) \rightarrow^{d} \dfrac{\modcZ}{Z^0} N(0,\mathrm{Var}_{\modcpi}(wg)).$$ 

\subsection{Proof of Theorem \ref{thm:concentrationChernoff}}
\label{subsec:concentrationChernoff} 
	Using equation $(1.2)$ in \cite{Lezaud01}, for $g$ such that $\norm{g}_{\infty} \leq a$, if we have
	$$t \geq \dfrac{32 a^2}{\lambda_2(-\modcM)\gamma^2} \ln \left(\dfrac{1}{\varepsilon \min_{x \in \mathcal{X}} \modcpi(x)}\right),$$
	then	
	\begin{align*}
	\mP_x\left(\dfrac{1}{t}\int_0^t w(\miY) g(\miY) \, ds - \modcpi(wg) \geq \gamma \right) &\leq \varepsilon, \\
	\mP_x\left( \dfrac{1}{t}\int_0^t w(\miY) g(\miY) \, ds - \dfrac{Z^0}{\modcZ} \mu(g) \geq \gamma \right) &\leq \varepsilon.	
	\end{align*}
	
	Specializing into $g = -1$ gives, for 
	$$t \geq \dfrac{32 }{\lambda_2(-\modcM)\gamma^2} \ln \left(\dfrac{1}{\varepsilon \min_{x \in \mathcal{X}} \modcpi(x)}\right),$$
	then 
	\begin{align*}
	\mP_x\left(  \dfrac{Z^0}{\modcZ} -\gamma \geq \dfrac{1}{t}\int_0^t w(\miY)  \, ds  \right) &\leq \varepsilon.
	\end{align*}
	From now on, without loss of generality, we assume that $g \geq 0$, and so
	\begin{align*}
	\mP_x&\left(\dfrac{\int_0^t w(\miY) g(\miY) \, ds}{\int_0^t w(\miY) \, ds} - \mu(g) \geq \dfrac{\gamma}{\dfrac{Z^0}{\modcZ} -\gamma}\right) \\
	&\leq \mP_x\left(\dfrac{\int_0^t w(\miY) g(\miY) \, ds}{\int_0^t w(\miY) \, ds} - \mu(g) \geq \dfrac{\gamma + \gamma\mu(g)}{\dfrac{Z^0}{\modcZ} -\gamma}\right) \leq \varepsilon.
	\end{align*}
	The desired result follows by taking
	$$\eta := \dfrac{\gamma}{\dfrac{Z^0}{\modcZ} -\gamma},$$
	and rearranging gives $$\gamma = \dfrac{Z^0}{\modcZ} \dfrac{\eta }{1  + \eta}.$$

\section{Time-dependent penalty $(\alpha_t)_{t \geq 0}$}\label{sec:timedpenalty}

In Section \ref{sec:constantpenalty}, we focus on the constant penalty parameter case: we run a homogeneous Markov chain $\modcX = (\modcX(t))_{t \geq 0}$ at a fixed $\alpha$, and correct the bias by proposing a self-normalized estimator. In this Section, instead of fixing $\alpha$, we anneal the penalty parameter by sending $\alpha_t \to 0$ as the time $t \to \infty$. This results in running a \textit{non-homogeneous} Markov chain $\modX = (\modX(t))_{t \geq 0}$ to approximately sample from $\mu$, where $\bm{\alpha} = (\alpha_t)_{t \geq 0}$. Unlike simulated annealing where the temperature parameter is annealed, we instead anneal the penalty parameter so that the state-dependent effect diminishes to zero in the long run, which allows for approximate sampling from the target distribution $\mu$.
%Let
%$$
%\modpi:=(\modpi(x))_{x\in\mathcal{X}},
%\quad
%\modpi(x):=\frac{e^{-\mcH_{1, \alpha_t, c}^{f}(x)}}{\sum_{y \in \mathcal{X}} e^{-\modH(y)} \triangleq Z_{1, \alpha_{t}, c}^{f}}
%$$

In the following, we shall denote, for $x \in \mathcal{X}$ and $A \subseteq \mathcal{X}$,
$$(\modP)^t(x,A) := e^{\int_0^t M^f_{1,\alpha_s,c} \, ds} (x,A) = \mathbb{P}_x(\modX(t) \in A).$$

Our first main result states that the non-homogeneous chain $\modX$ is weakly ergodic as long as $\alpha'_{t} \to 0$. In other words, if the change of $\alpha_t$ approaches $0$ as $t \to \infty$, the non-homogeneous chain tracks closely its instantaneous stationary distribution $\modpi$:
\begin{theorem}[Weak ergodicity of $\modX$]\label{thm:TVmixingMod}
	 If $\alpha'_{t} \to 0$ as $t\to \infty$, then for any $x \in \mathcal{X}$,
	\begin{equation}\label{eq:TVmixingMod}
	\left\|\left(\modP\right)^t(x,\cdot)-\modpi \right\|_{TV}\to 0  \text{ as } t\to \infty.
	\end{equation}
\end{theorem}
The proof of Theorem \ref{thm:TVmixingMod} is given in Section \ref{subsec:Proof_TVmixingMod}. 

%For large  enough $c$ such that $\mcH(x)\leq c$ for $\forall x\in \mathcal{X}$,
%\[\modpi(x)=
%\mu (x)=\frac{e^{-\mcH(x)}}{\sum_{y \in \mathcal{X}} e^{-\mcH(y)}}=\frac{e^{-(\mcH(x)-\mcH_{\min})}}{\sum_{y\in\mathcal{X}}e^{-(\mcH(y)-\mcH_{\min})}\overset{\triangle}{=}Z } .
%\]

If we impose the additional assumption that $\alpha_t \downarrow 0$, then the non-homogeneous chain $\modX$ is strongly ergodic and converges to the desired target distribution $\mu$ in total variation distance:
\begin{theorem}[Strong ergodicity of $\modX$ ]\label{thm:TVmixingStationary}
	If $\alpha_t\downarrow0,~\alpha'_{t} \to 0$ as $t\to\infty$, then for any $x \in \mathcal{X}$,
	\[
	\left\|\left(\modP\right)^t(x,\cdot)-\mu  \right\|_{TV} \to 0\text{ as } t\to\infty.	\]
%	where  $$\|\modpi-\mu  \|_{TV}=\mathcal{O}(\alpha_t)$$ and \begin{align*}
%	\|\left(\modP\right)^t(x,\cdot)-\modpi\|_{TV}  \leq \mathcal{O}\left(e^{-\frac12\int_0^t\zeta_sds}\right)
%	\end{align*}with  $\zeta_t=2\lambda_{2}(-\modM) - \gamma_t$.
\end{theorem}

Unlike simulated annealing where logarithmic cooling schedule of the temperature is commonly employed to ensure strong ergodicity \cite{C21MPRF,Miclo92AIHP,M18}, in this context the penalty parameter can decay to zero at a faster rate than logarithmic cooling while maintaining strong ergodicity. For instance, we can adopt the exponential annealing schedule by taking $\alpha_t = e^{-t}$. With this choice, we note that the conditions in Theorem \ref{thm:TVmixingStationary} are satisfied.

The proof of Theorem \ref{thm:TVmixingStationary} is deferred to Section \ref{subsec:Proof_TVmixingStationary}.

%\begin{rk}
%	If $\int_0^\infty \zeta_tdt=\infty$, then $\lim_{t\to\infty}F_t(x)=0$, for any $x\in\mathcal{X}$.
%\end{rk}

In our next result, we quantify the convergence rate of the total variation mixing by specializing into the case where $f(x) = x^2$. The primary reason for considering quadratic $f$ is that the reduced critical height can be bounded independently of the system size, leading to rapid mixing in the examples to be presented in Section \ref{sec:apps}.

We would be interested in estimating \textit{the total variation mixing time} $t_{mix}^f(\varepsilon)$ of the non-homogeneous chain $\modX$, defined to be for $\varepsilon > 0$,
\begin{align}\label{eq:nonhomotmix}
	t_{mix}^f(\varepsilon) := \inf \bigg\{t \geq 0;~ 	\max_{x \in \mathcal{X}} \left\|\left(\modP\right)^t(x,\cdot)-\mu \right\|_{TV} \leq \varepsilon\bigg\}.
\end{align}

Note that the dependence on $\bm{\alpha}$ and $c$ are suppressed on the left hand side.
 
\begin{theorem}[Total variation mixing time with $f(x) = x^2$ and logarithmic annealing schedule]\label{thm:Mixingx2}
	Suppose that $ f(x)=x^2$, where we recall $\modcH$ is computed in \eqref{eq:Hal_x2}. For a large enough universal constant $C > 0$ that is independent of $\mathcal{X}$ or $\mcH$, denote 
	$$
	\beta_{t}:=\frac1{\alpha_{t}}=\frac{2}{3\pi} \ln (e^{\frac{3\pi}{2}}+p t), \quad p:=\frac{\pi K}{\mathcal{M}}, \quad \mathcal{M} := C\left((\ln |\mathcal{X}|)\vee \left(\mcH_{\max }-\mcH_{\min }\right)\right).
	$$
	For $t \geq 0$, we employ the following logarithmic annealing schedule on the penalty parameter with
	$$\alpha_t := \frac{3\pi}{2} \dfrac{1}{\ln (e^{\frac{3\pi}{2}}+p t)}.$$
	Assume further that there exists a positive constant $K = K(c)$ such that 
	\begin{equation}\label{eq:conditionK}
		\lambda_{2}(-\modM)\geq Ke^{-\frac{3\pi}{2}\sqrt{\beta_t}}\geq Ke^{-\frac{3\pi}{2}\beta_t},
	\end{equation} 
	\textcolor{black}{and we choose $c$ such that $c - \mcH_{\min} = \delta = \mathcal{O}(1)$.}	For any $x \in \mathcal{X}$, we have 
	$$
	\left\|\left(\modP \right)^t(x,\cdot)-\mu\right\|_{TV} = \textcolor{black}{\mathcal{O}(|\mathcal{X}| e^{\mathcal{M}}e^{- \mathcal{M} \beta_{t}}).}
	$$
	In particular, 
	\textcolor{black}{
	\begin{align*}
		t_{mix}^f(\varepsilon) &= \mathcal{O}\left( \dfrac{1}{p} \left(\exp\bigg\{\dfrac{3\pi( \ln |\mathcal{X}| + \mathcal{M} + \log(1/\varepsilon))}{2\mathcal{M}}\bigg\} - \exp\{3\pi/2\}\right)\right) \\
		&= \mathcal{O}\left(\dfrac{(\ln |\mathcal{X}|)\vee \left(\mcH_{\max }-\mcH_{\min }\right)}{K} \left(\dfrac{1}{\varepsilon}\right)^{\frac{3\pi}{2\left((\ln |\mathcal{X}|)\vee \left(\mcH_{\max }-\mcH_{\min }\right)\right)}}\right).
	\end{align*}}
\end{theorem}
The proof of Theorem \ref{thm:Mixingx2} is given in Section \ref{subsec:Proof_Mixingx2}.

 \subsection{Proof of Theorem \ref{thm:TVmixingMod}}
 \label{subsec:Proof_TVmixingMod} 
 	In view of \cite[Theorem $A$]{Gidas85}, it suffices to find $(\gamma_t)_{t \geq 0}$ that satisfies \begin{equation}\label{eq:conditionBounddpi}
 	\left|\frac{d}{d t} \modpi(x)\right| \leqslant
 	\gamma_{t} \cdot \modpi(x),
 	\end{equation} where  $\gamma_t$ is independent of $x$ and fulfills 
 	\begin{equation}\label{eq:conditionGammaSpegap}
 	\lim _{t \rightarrow \infty} \frac{\gamma_t}{\lambda_2(-\modM)}=0
 	\end{equation}
 	with 
 	\begin{align}\label{eq:conditionspegap}
 		\int_{0}^{ \infty} \lambda_2(-\modM) d t=+\infty.
 	\end{align}
 	
 	First, \eqref{eq:conditionspegap} is satisfied due to Corollary \eqref{cor:spectralgap}. We proceed to show \eqref{eq:conditionBounddpi}. Denote the negative time derivative of $\modH$ as
 	$$  
 	G_{1,\alpha_t,c}^f(x):=- \frac{d}{dt}\mcH_{1,\alpha_t,c}^f(x)\leq 0.$$
 	We observe that
 	$$
 	G_{1,\alpha_t,c}^f(x)=\int_{\mcH_{\min}}^{\mcH(x)}\frac{\alpha'_tf((u-c)_+)}{(\alpha_tf\left((u-c)_{+}\right)+1)^2}du \geq \alpha_t'f(\mcH_{\max}-c)(\mcH_{\max}-\mcH_{\min})
 	$$
 	and
 	\begin{align*} 
 	\frac{d}{dt}\modpi(x) &=
 	\frac1{\modZ} \frac{d}{dt}\left(e^{-\mcH_{1,\alpha_t,c}^f(x)}\right)+e^{-\mcH_{1,\alpha_t,c}^f(x)}\frac{d}{dt}\left( \frac1{\modZ}\right)\\ 
 	&=G_{1,\alpha_t,c}^f(x)\cdot \modpi(x)-e^{-\mcH_{1,\alpha_t,c}^f(x)} \left( \frac1{\modZ}\right)^2 \sum_{x \in \mathcal{X}}   e^{-\modH(x)}\cdot G_{1,\alpha_t,c}^f(x)\\
 	&=G_{1,\alpha_t,c}^f(x)\cdot \modpi(x)-\modpi(x)    \sum_{x \in \mathcal{X}}  \modpi(x)  \cdot G_{1,\alpha_t,c}^f(x),
 	\end{align*}
 	which implies
 	\begin{align*}
 	\left| \frac{d}{dt}\modpi(x)  \right|
 	&\leq \modpi(x) \left(\max_x G_{1,\alpha_t,c}^f(x) - \min_x G_{1,\alpha_t,c}^f(x)\right) \notag\\
	&\leq -\alpha_t'f(\mcH_{\max}-c)(\mcH_{\max}-\mcH_{\min})\cdot \modpi(x).
 	\end{align*}
%  	\begin{align}
% 	\left| \frac{d}{dt}\modpi(x)  \right|
% 	&\leq \left|G_{1,\alpha_t,c}^f(x)\cdot \modpi(x)\right|+\left|-\modpi(x)    \sum_{x \in \mathcal{X}}  \modpi(x)  \cdot G_{1,\alpha_t,c}^f(x)\right|\notag\\
% 	&\leq  -\alpha_t'f(\mcH_{\max}-c)(\mcH_{\max}-\mcH_{\min})\modpi(x)\cdot\left( 1+ \sum_{x\in\mathcal{X}}\modpi(x) \right)\notag\\
% 	&{=} -2\alpha_t'f(\mcH_{\max}-c)(\mcH_{\max}-\mcH_{\min})\cdot \modpi(x).\label{eq:bounddpi}
% \end{align}
 	Then, \eqref{eq:conditionBounddpi} is satisfied with 
 	$\gamma_t=-\alpha_t'f(\mcH_{\max}-c)(\mcH_{\max}-\mcH_{\min})$  and \eqref{eq:conditionGammaSpegap} is satisfied since $ \alpha'_{t} \to 0$ as $t\to \infty$ and the time-independent lower bound of the spectral gap as in Corollary \eqref{cor:spectralgap}. Thus, \eqref{eq:TVmixingMod} is proved by applying \cite[Theorem A]{Gidas85}.

 \subsection{Proof of Theorem \ref{thm:TVmixingStationary}}\label{subsec:Proof_TVmixingStationary} 
 	Using the triangle inequality, we have\[
 \left\|\left(\modP\right)^t(x,\cdot)-\mu  \right\|_{TV}\leq \left\|\left(\modP\right)^t(x,\cdot)-\modpi\right\|_{TV}+\|\modpi-\mu  \|_{TV}\to 0\text{ as } t\to\infty	\]where the first term converges to $0$ due to Theorem \ref{thm:TVmixingMod}, and we recall from Theorem \ref{thm:ConvergenceTradeoff} that the second term is $$\|\modpi-\mu  \|_{TV}=\mathcal{O}(\alpha_t).$$
\subsection{Proof of Theorem \ref{thm:Mixingx2}}
\label{subsec:Proof_Mixingx2} 
We follow the same roadmap as the proof of \cite[Appendix A.1]{DM94}. First, using the Cauchy-Schwartz inequality, we have
\begin{align*} 
\left\|\left(\modP \right)^t(x,\cdot)-\mu\right\|_{TV}
&\notag=\frac12 \sum_{y \in \mathcal{X}}\left|\left(\modP\right)^t(x,y)-\mu(y)\right|\\
&\leq\frac12 \textcolor{black}{\sqrt{\sum_{y \in \mathcal{X}}\left(\frac{\left(\modP\right)^t(x,y)-\mu(y)}{\modpi(y)} \right)^2\modpi(y)} =:\frac12 \sqrt{u_t(x)}}
\end{align*}
where  
\[
u_t(x):=\sum_{y \in \mathcal{X}} \left(g_t(x,y)-\frac{\mu(y)}{\modpi(y)}\right)^2\modpi(y),\quad g_t(x,y):=\frac{\left(\modP\right)^t(x,y)}{\modpi(y)}.
\]

	Next, we calculate the time derivative of $u_t$:
	\begin{align}\notag
	\frac{du_t(x)}{dt}
	&=\frac{d}{dt}\left( \sum_{y\in\mathcal{X}}\left(g_t(x,y)-\frac{\mu(y)}{\modpi(y)}\right)^2\modpi(y) \right) \\\notag
	&=\sum_{y \in \mathcal{X}} 2\left(g_t(x,y)-\frac{\mu(y)}{\modpi(y)}\right)\left(\modpi(y)\frac{dg_t(x,y)}{dt}+ \frac{\mu(y) }{\modpi(y)}\frac{d\modpi(y)}{dt}\right)\\\notag
	&\qquad \qquad \qquad\qquad +\sum_{y\in\mathcal{X}}\left( g_t(x,y)-\frac{\mu(y)}{\modpi(y)}\right)^2\frac{d\modpi(y)}{dt}\\\notag
	&=\sum_{y \in \mathcal{X}}2\left(g_t(x,y)-\frac{\mu(y)}{\modpi(y)}\right)\left( \modM g_t(x,y) -\frac{g_t(x,y)}{\modpi(y)}\frac{d\modpi(y)}{dt}\right)\modpi(y)\\\notag
	&\qquad +\sum_{y \in \mathcal{X}}2\left(g_t(x,y)-\frac{\mu(y)}{\modpi(y)}\right)\frac{\mu(y)}{\modpi(y)}\cdot\frac{d\modpi(y)}{dt} +\sum_{y \in \mathcal{X}}\left(g_t(x,y)-\frac{\mu(y)}{\modpi(y)}\right)^2\frac{d\modpi(y)}{dt} \\ \label{eq:du2t}
	&=-2\langle -\modM g_t,g_t\rangle_{\modpi}-\sum_{y \in \mathcal{X}}\left(g_t(x,y)-\frac{\mu(y)}{\modpi(y)} \right)^2\frac{d\modpi(y)}{dt},
	\end{align}  
	where the third equality follows from \cite[Lemma A.1.1]{DM94} and the reversibility of $M^f_{1,\alpha_t,c}$ with respect to $\modpi$.
	
	In the special case where $ f(x)=x^2$, $\modcH$ is computed as in \eqref{eq:Hal_x2}, we can then calculate the derivative of $\modpi$ with respect to $\beta = 1/\alpha$: 
	\begin{align} 
	\frac{d}{d\beta}\modcpi(y) 
	=\modcpi(y)\left(-\frac{d\modcH(y)}{d\beta} \right)
	-\modcpi(y)\sum_{x \in \mathcal{X}} \modcpi(x)\left(- \frac{d\modcH(x)}{d\beta} \right)\label{eq:dmodpitobeta}
	\end{align}
	with 
	$$
	-\frac{d\modcH(y)}{d\beta}  =\left\{\begin{array}{ll}
	0, &H(y) \leqslant c, \\  \dfrac{  \mcH(y)-c}{2(\mcH(y)-c)^2+2\beta }-\dfrac{\text{arctan} \left(\sqrt{\frac1\beta}\left(
		\mcH(y)-c\right)\right)}{2\sqrt{\beta}} ,& H(y)>c.
	\end{array}\right.
	$$
	We also have the bounds that
	\[
	\frac{-\pi}{4\sqrt{\beta}}\leq -\frac{d\modcH(y)}{d\beta}\leq \frac{1}{2\beta}\left(\mcH_{\max}-\mcH_{\min} \right),
	\] 
	which when combine with \eqref{eq:dmodpitobeta} implies 
	\begin{align}\label{eq:Bounddlnmodpi2beta}
	\left|\frac{d}{d\beta}\ln \modcpi(y) \right| \leq  \max_{y\in\mathcal{X}}\left(-\frac{d\modcH(y)}{d\beta}\right)  - \min_{x \in \mathcal{X}}   \left(-\frac{d\modcH(x)}{d\beta}\right) = \mathcal{O}\left(\frac{1}{\sqrt{\beta}}\left(\mcH_{\max}-\mcH_{\min} \right)\right).
	\end{align}
	
	With the choice of $K, \beta_t$ and $\mathcal{M}$ in the theorem, we have
	$$Ke^{-\frac\pi2\beta_t}-\frac{\mathcal{M}}{2}\frac{d\beta_t}{dt} = \mathcal{M} \dfrac{d \beta_t}{dt},$$
	and using \eqref{eq:du2t} and \eqref{eq:conditionK} give
	\begin{equation*}\label{eq:boundu2t}
	\frac{du_t(x)}{dt} \leq -2\left(Ke^{-\frac\pi2\beta_t}-\frac{\mathcal{M}}{2}\frac{d\beta_t}{dt} \right)u_t(x)=-2\mathcal{M}\cdot\frac{d\beta_t }{dt}\cdot u_t(x)
	\end{equation*}
	which implies that 
	\begin{equation*}\label{eq:boundu}
	u_t(x)\leq u_0(x)e^{-2\mathcal{M}(\beta_t - 1)}  
	\end{equation*} 
	with  \begin{align*}
	u_{0}(x)
	&=\left(\frac{1-\mu(x)}{\mu_{1, \alpha_{0}, c}^{f}(x)}\right)^{2} \mu_{1, \alpha_{0}, c}^{f}(x)
	+\sum_{y \neq x}\left(\frac{\mu(y)}{\mu_{1, \alpha_{0}, c}^{f}(y)}\right)^{2} \mu_{1, \alpha_{0}, c}^{f}(y)\\
	&\leq \frac{1}{\min_{x \in \mathcal{X}}\mu_{1,\alpha_0,c}^f(x)}+ \max_{x\in\mathcal{X}}\left(\frac{\mu(x)}{{\mu}_{1, \alpha_{0}, c}^{f}(x)}\right)^2\\
	&\leq e^{c-\mcH_{\min}+\frac{1}{\sqrt{\alpha_{0}}\cdot }\frac{\pi}{2}}+ \left( \frac{Z_{1, \alpha_{0}, c}^{f}}{Z^0}\right)^2
	\leq \mathcal{O}(|\mathcal{X}|^2),
	\end{align*}
	\textcolor{black}{where the last inequality follows from the assumption that $c - \mathcal{H}_{\min} = \delta = \mathcal{O}(1)$.} The desired result follows.
 
\section{Applications}\label{sec:apps}

In this Section, we shall apply the \textit{model-independent} mixing time results that we obtain in Section \ref{sec:constantpenalty} and Section \ref{sec:timedpenalty} to some common statistical physics models where the Glauber dynamics is utilized for sampling from the target Gibbs distribution $\mu$. In particular, we will be interested in comparing the following mixing parameters: for any $\eta, \varepsilon > 0$, $g\in\ell^2(\mu)$
\begin{itemize}
	\item (Large deviation from the target mean) Recall that in \eqref{eq:conctime}, we introduce
	$$t_{conc}(\eta, \varepsilon) = \inf \bigg\{t \geq 0;~ 	\sup_{g \in \ell^2(\mu);~\norm{g}_{\infty} \leq a} \sup_{x \in \mathcal{X}}\mP_x\left( \dfrac{\int_0^t w(\miY) g(\miY) \, ds}{\int_0^t w(\miY) \, ds} -  \mu(g) \geq \eta \right) \leq \varepsilon\bigg\}.$$
	
	\item (Total variation mixing time of the modified dynamics with $f(x) = x^2$ and logarithmic annealing schedule of the penalty parameter) Recall that in \eqref{eq:nonhomotmix}, we introduce
	$$t_{mix}^f(\varepsilon) = \inf \bigg\{t \geq 0;~ 	\max_{x \in \mathcal{X}} \left\|\left(\modP\right)^t(x,\cdot)-\mu  \right\|_{TV} \leq \varepsilon\bigg\}.$$
	
	\item (Total variation mixing time of the original dynamics)
	Recall that in \eqref{eq:mixingtime}, we define the total variation mixing time of the MH dynamics that target $\mu$ to be
	$$t_{mix}^0(\varepsilon) = \inf \bigg\{t \geq 0;~ 	\max_{x \in \mathcal{X}} \left\|\left(P^0\right)^t(x,\cdot)-\mu  \right\|_{TV} \leq \varepsilon\bigg\}.$$
\end{itemize}

We shall investigate the ferromagnetic Ising model on the $n$-dimensional hypercube in Section \ref{subsec:app_hypercube}, the ferromagnetic Ising model on the complete graph in Section \ref{subsec:app_complete}, and finally the $q$-state Potts model on the two-dimensional torus in Section \ref{subsec:app_Potts}.

\subsection{Ferromagnetic Ising model on the $n$-dimensional hypercube}\label{subsec:app_hypercube}

In this Section, we specialize into the case of the Ising model on a finite connected graph $G = (V,E)$. Let $\mathcal{X} = \{-1,1\}^V$ be the set of possible configurations $\sigma = \{\sigma(v);~v \in V\}$ in which we associate each vertex $v \in V$ with a spin value $\sigma(v) \in \{-1,1\}$. The Hamiltonian is given by
\begin{align}\label{eq:Hamhypercube}
	\mcH(\sigma) = - \dfrac{J}{2} \sum_{(v,w) \in E} \sigma(v) \sigma(w) - \dfrac{h}{2} \sum_{v \in V} \sigma(v),
\end{align}
where $J = \mathcal{O}(1) > 0$ is the ferromagnetic pair potential and $h = \mathcal{O}(1) > 0$ is the external magnetic field. We also denote by $\mathbf{+1}$ (resp.~ $\mathbf{-1}$) to be the all-plus (resp.~all-minus) configuration, where we recall that $\mathbf{+1}$ is the ground-state or stable configuration while $\mathbf{-1}$ is the metastable configuration of $\eqref{eq:Hamhypercube}$. 

For the Glauber dynamics, we consider a simple random walk proposal with transition matrix $P^{SRW} = (P^{SRW}(\sigma_1,\sigma_2))_{\sigma_1,\sigma_2 \in \mathcal{X}}$ given by
\begin{align*}
	P^{SRW}(\sigma_1,\sigma_2) &:= \dfrac{1}{n} \1_{\{\textrm{there exists } i \textrm{ such that } \sigma_1(i) = -\sigma_2(i) \textrm{ and } \sigma_1(j) = \sigma_2(j) \textrm{ for all }j \neq i\}}.
\end{align*}
The stationary distribution of $P^{SRW}$ is the uniform distribution on $\mathcal{X}$. 

%\subsubsection{Improved mixing of the Glauber dynamics on the $n$-dimensional hypercube}

\begin{theorem}\label{thm:MixingGlauber_hypercube}
	Suppose that $G = (V,E)$ is an $n$-dimensional hypercube. Consider the Glauber dynamics on $\mathcal{X} = \{-1,1\}^V$ with proposal generator $P^{SRW} - I$, target distribution $\modcpi$ with Hamiltonian given in \eqref{eq:Hamhypercube}, $f(x) = x^2$ and $c = \mcH(\mathbf{+1}) + \delta$, where $\delta > 0$. Given $\eta, \varepsilon > 0$, we have
	\begin{enumerate}
		\item\label{it:hypercube1}  for $g \in \ell^2(\mu)$ such that $\norm{g}_{\infty} \leq a$,
		$$t_{conc}(\eta, \varepsilon) = \mathcal{O}\left(n^2 2^n e^{3 \delta + \frac{3}{\sqrt{\alpha}}\frac{\pi}{2}} \dfrac{a^2(1+\eta)^2}{\eta^2} \ln \left(\dfrac{1}{\varepsilon}\right)\right).$$
		
		\item\label{it:hypercube2}[Rapid mixing: mixing time is at most polynomial in the number of vertices in the modified Glauber dynamics with time-dependent penalty] Assume the logarithmic annealing schedule $\alpha_t = \frac{3 \pi}{2} \dfrac{1}{\ln (e^{\frac{3\pi}{2}}+p t)}$, where $p$ is chosen as in Theorem \ref{thm:Mixingx2}, then
		\textcolor{black}{$$t_{mix}^f(\varepsilon) = \mathcal{O}\left(n^2 2^n e^{3 \delta} \left(\dfrac{1}{\varepsilon}\right)^{\frac{3\pi}{2n \ln 2}}\right).$$}
		
		\item\label{it:hypercube3}[Torpid mixing: mixing time is at least exponential in the number of vertices in the original Glauber dynamics]
		Assume in addition that $0 \leq h/J \leq n$. For constant $\mathbb{M} = \mathbb{M}(J,h)$ that does not depend on $n$,
		$$t_{mix}^0(\varepsilon) = \Omega\left(e^{\mathbb{M} 2^n}\ln \left(\dfrac{1}{2 \varepsilon}\right)\right).$$
	\end{enumerate}
\end{theorem}

\begin{rk}[On the choice of $c$]\label{rk:cchoicehypercube}
	We notice that in both item \eqref{it:hypercube1} and item \eqref{it:hypercube2}, the upper bounds there depend on $\delta$ through $\mathcal{O}(e^{3\delta})$. Ideally, the parameter $c$ should be chosen to be as close as $\mathcal{H}(\mathbf{+1})$ so that $\delta$ is small enough. We also utilize the fact that the ground-state configuration in this setting is the all-ones configuration $\mathbf{+1}$.
\end{rk}

We observe that the original Glauber dynamics has an exponentially large mixing time in the system size $2^n$, while the proposed annealing algorithm mixes rapidly in the sense that the mixing time is at most polynomial in $2^n$.

\subsubsection{Proof of Theorem \ref{thm:MixingGlauber_hypercube}}

In this subsection, we give the proof of Theorem \ref{thm:MixingGlauber_hypercube}. We first state a spectral gap lower bound result that will be used in subsequent sections:

\begin{lemma}\label{lem:Glauberspectrapgap}
	Given a graph $G = (V, E)$, consider the Glauber dynamics on $\mathcal{X} = \{-1,1\}^V$ with proposal generator $P^{SRW} - I$, target distribution $\modcpi$ and denote $n = |V|$, then 
	\begin{align*}
		\lambda_2(-\modcM) \geq \dfrac{2}{n} e^{-3 \max_{\sigma} \modcH(\sigma)}.
	\end{align*}
	In particular, if we take $f(x) = x^2$ and $c = \mcH(\mathbf{+1}) + \delta$, where $\delta > 0$, we have
	\begin{align*}
		\lambda_2(-\modcM) \geq \dfrac{2}{n} e^{- 3 \delta - \frac{3}{\sqrt{\alpha}}\frac{\pi}{2}}.
	\end{align*}
\end{lemma}

%We proceed to give a proof of Lemma \ref{lem:Glauberspectrapgap}.

\begin{proof}[Proof of Lemma \ref{lem:Glauberspectrapgap}]
	To prove the desired result, we shall invoke the following classical stability result of the spectral gap by \cite{HS87}:
	\begin{lemma}\label{lem:relationSpeGapL1L2}
		Let $(\mathcal{X}, d, \mu)$ be a finite metric measure space and suppose $d \nu=\frac{e^{-U(x)}}{Z} d \mu$. Suppose that $L_1$ is a reversible generator with respect to $\nu$, and that $L_2$ is another reversible generator with respect to $\mu .$ Suppose furthermore that $A L_2(x, y) \leq L_1(x, y)$ for all $x,y \in \mathcal{X}$. Then if $L_2$ has spectral gap $\lambda_2(-L_2)$, the spectral gap of $L_1$ satisfies
		$$
		A e^{-2(\max U-\min U)} \lambda_2(-L_2) \leq \lambda_2(-L_1).
		$$
	\end{lemma}
	The desired result follows if we take $L_1 =\modcM$, $L_2 = P^{SRW}-I$ and $A = \exp\bigg\{-\max_{\sigma}\modcH(\sigma) \bigg\},$
	and recall that in \cite[Example $3.2$]{DSC96} the spectral gap of $P^{SRW}$ is given by $\lambda_2(-(P^{SRW}-I)) = 2/n$. 
	
	In the special case when $f(x) = x^2$ and $c = \mcH(\mathbf{+1}) + \delta$, we utilize Section \ref{subsec:quadraticf} and the fact that $\sup_x \arctan(x) = \pi/2$.
	
	%	The above computation leads to
	%	\begin{align*}
		%	\lambda_2(-\modcM) &\geq 2 N^{-1} e^{-3(\max_{\sigma \in \Sigma_N} \mcH^{f_{2}}_{T_t,c_N}(\sigma) - \min_{\sigma \in \Sigma_N} \mcH^{f_{2}}_{T_t,c_N}(\sigma))} \\
		%	&\geq 2 N^{-1} e^{-\beta_t 3 \left(2 \epsilon_N + \frac{\pi}{2}\right)}.
		%	\end{align*}
\end{proof}

We now give the proof of Theorem \ref{thm:MixingGlauber_hypercube}.

\begin{proof}[Proof of Theorem \ref{thm:MixingGlauber_hypercube}]
	First, we prove item \eqref{it:hypercube1}. The desired result follows from Theorem \ref{thm:concentrationChernoff} and the spectral gap lower bound Lemma \ref{lem:Glauberspectrapgap} as well as the inequality that
	$$\dfrac{\modcZ}{Z^0} \leq \modcZ \leq 2^n.$$
	
	Next, we prove item \eqref{it:hypercube2}, which simply follows from Theorem \ref{thm:Mixingx2} and the observation from Lemma \ref{lem:Glauberspectrapgap} that
	$$\mcH_{\max }-\mcH_{\min } = \mathcal{O}(n 2^n), \quad K = \dfrac{2}{n} e^{-3 \delta}.$$
	
	Finally, we prove item \eqref{it:hypercube3}. Using \cite[Lemma $20.11$]{LPW17}, we see that
	\begin{align*}
		t_{mix}^0(\varepsilon) \geq \dfrac{1}{\lambda_2(-M^0)} \ln \left(\dfrac{1}{2 \varepsilon}\right).
	\end{align*}
	According to \cite{J17,DHJN17}, the critical height $m^0$ is bounded below by
	$$m^0 \geq \mathbb{M} 2^n,$$
	and the desired result follows from \cite[Lemma $2.3$]{HS88} since 
	$$\lambda_2(-M^0) = \mathcal{O}(e^{\mathbb{M} 2^n}/4^n) = \mathcal{O}(e^{\mathbb{M} 2^n}).$$
	Note that in the above equations the value of $\mathbb{M}$ may possibly be changed in each line.
\end{proof}

\subsection{Ferromagnetic Ising model on the complete graph}\label{subsec:app_complete}

In the second example, we consider the ferromagnetic Ising model on the complete graph $G = K_n$ with $n = |V|$ vertices. The Hamiltonian remains to be \eqref{eq:Hamhypercube}, except now the underlying graph structure is $K_n$.

\begin{theorem}\label{thm:MixingGlauber_complete}
	Suppose that $G = K_n$ is the complete graph on $n$ vertices. Consider the Glauber dynamics on $\mathcal{X} = \{-1,1\}^V$ with proposal generator $P^{SRW} - I$, target distribution $\modcpi$ with Hamiltonian given in \eqref{eq:Hamhypercube}, $f(x) = x^2$ and $c = \mcH(\mathbf{+1}) + \delta$, where $\delta > 0$. Given $\eta, \varepsilon > 0$, we have
	\begin{enumerate}
		\item\label{it:complete1}  For $g \in \ell^2(\mu)$ such that $\norm{g}_{\infty} \leq a$,
		$$t_{conc}(\eta, \varepsilon) = \mathcal{O}\left(n^2 2^n e^{3 \delta + \frac{3}{\sqrt{\alpha}}\frac{\pi}{2}} \dfrac{a^2(1+\eta)^2}{\eta^2} \ln \left(\dfrac{1}{\varepsilon}\right)\right).$$
		
		\item\label{it:complete2}[Rapid mixing: mixing time is at most polynomial in the number of vertices in the modified Glauber dynamics with time-dependent penalty] Assume the logarithmic annealing schedule $\alpha_t = \frac{3 \pi}{2} \dfrac{1}{\ln (e^{\frac{3\pi}{2}}+p t)}$, where $p$ is chosen as in Theorem \ref{thm:Mixingx2}, then
		\textcolor{black}{$$t_{mix}^f(\varepsilon) = \mathcal{O}\left(n^3 e^{3 \delta} \left(\dfrac{1}{\varepsilon}\right)^{\frac{3\pi}{2n \ln 2}}\right).$$}
		
		\item\label{it:complete3}[Torpid mixing: mixing time is at least exponential in the number of vertices in the original Glauber dynamics]
		Assume in addition that $h/J$ is not an integer. For constant $\mathbb{M} = \mathbb{M}(J,h)$ that does not depend on $n$,
		$$t_{mix}^0(\varepsilon) = \Omega\left(e^{\mathbb{M} n^2}\ln \left(\dfrac{1}{2 \varepsilon}\right)\right).$$
	\end{enumerate}
\end{theorem}

\begin{rk}[On the choice of $c$]
	Similar to Remark \ref{rk:cchoicehypercube}, ideally the parameter $c$ should be chosen to be as close as $\mathcal{H}(\mathbf{+1})$ so that $\delta$ is small enough, since in both item \eqref{it:complete1} and item \eqref{it:complete2}, the upper bounds there depend on $\delta$ through $\mathcal{O}(e^{3\delta})$. 
\end{rk}

Similar to the previous example where we consider the ferromagnetic Ising model on the hypercube, we observe that the original Glauber dynamics has an exponentially large mixing time in the system size $n$, while the proposed annealing algorithm mixes rapidly with its mixing time being at most polynomial in $n$. As both the hypercube and the complete graph belong to the class of expander graphs, the original critical height grows at least linearly with $n$. The advantage of the proposed annealing algorithm is that the reduced critical height is bounded independently of $n$, and consequently this leads to rapid mixing.

\subsubsection{Proof of Theorem \ref{thm:MixingGlauber_complete}}

	The proof is similar to the proof of Theorem \ref{thm:MixingGlauber_hypercube}. First, we prove item \eqref{it:complete1}. Making use of Theorem \ref{thm:concentrationChernoff}, the spectral gap lower bound Lemma \ref{lem:Glauberspectrapgap} and the inequality that
	$$\dfrac{\modcZ}{Z^0} \leq \modcZ\leq 2^n,$$
	the desired result follows.
	
	Next, we prove item \eqref{it:complete2}, which readily follows from Theorem \ref{thm:Mixingx2} and  Lemma \ref{lem:Glauberspectrapgap} that
	$$\mcH_{\max }-\mcH_{\min } = \mathcal{O}(n^2), \quad K = \dfrac{2}{n} e^{-3 \delta}.$$
	
	Finally, we prove item \eqref{it:complete3}. Using \cite[Lemma $20.11$]{LPW17}, we see that
	\begin{align*}
		t_{mix}^0(\varepsilon) \geq \dfrac{1}{\lambda_2(-M^0)} \ln \left(\dfrac{1}{2 \varepsilon}\right).
	\end{align*}
	According to \cite{DHJN17}, the critical height $m^0$ is bounded below by
	$$m^0 \geq  \mathbb{M} n^2,$$
	and the desired result follows from \cite[Lemma $2.3$]{HS88} since 
	$$\lambda_2(-M^0) = \mathcal{O}(e^{\mathbb{M} n^2}/4^n) = \mathcal{O}(e^{\mathbb{M} n^2}).$$
	Note that in the above equations the value of $\mathbb{M}$ may possibly be changed in each line.

\subsection{$q$-state Potts model on the two-dimensional torus}\label{subsec:app_Potts}

The $q$-state Potts model is a generalization of the Ising model that we saw in the previous two sections, which is a spin system characterized by a set $S=\{1, \ldots, q\}$ of spin values on a finite graph $G=(V, E)$ with $q \in \mathbb{N}$. In this Section, we shall specialize into the two-dimensional torus $G = \left(\mathbb{Z}/ n\mathbb{Z}\right)^2$. Each configuration $\sigma \in \mathcal{X}=S^{V}$ assigns a spin value $\sigma(v) \in S$ to each vertex $v \in V$, and the Hamiltonian is given by
\begin{align}\label{eq:Ham_Potts}
	\mcH(\sigma):=-J \sum_{(v, w) \in E} \1_{\{\sigma(v)=\sigma(w)\}}, \quad \sigma \in \mathcal{X},
\end{align}
where $J = \mathcal{O}(1) > 0$ is the coupling constant. In subsequent analysis, we shall be interested in monochromatic configurations. Denote $\mathbf{m}$ to be the spin configuration in which $\sigma(v) = m$ for all $v \in V$, where $m \in S$. Unlike the previous two examples of ferromagnetic Ising model where $\mathbf{+1}$ is the only ground-state configuration, we note that the $q$ monochromatic configurations are the ground-state configurations in the $q$-state Potts model.

For the Glauber dynamics, we consider a simple random walk proposal with transition matrix $P^{Potts} = (P^{Potts}(\sigma_1,\sigma_2))_{\sigma_1,\sigma_2 \in \mathcal{X}}$ given by
\begin{align*}
	P^{Potts}(\sigma_1,\sigma_2) &:= \dfrac{1}{q |V|} \1_{\{|v \in V;~ \sigma_1(v) \neq \sigma_2(v)| = 1\}}.
\end{align*}
The stationary distribution of $P^{Potts}$ is the uniform distribution on $\mathcal{X}$. We now state the main result of this Section:

\begin{theorem}\label{thm:MixingGlauber_Potts}
		Suppose $G = \left(\mathbb{Z}/ n\mathbb{Z}\right)^2 = (V, E)$, consider the Glauber dynamics on $\mathcal{X} = \{1,2,\ldots,q\}^V$ with proposal generator $P^{Potts} - I$, target distribution $\modcpi$ with Hamiltonian given in \eqref{eq:Ham_Potts}, $f(x) = x^2$ and $c = \mcH(\mathbf{1}) + \delta$. Given $\varepsilon > 0$, we have
	\begin{enumerate}
%		\item\label{it:Potts1}  For $g \in \ell^2(\mu)$ such that $\norm{g}_{\infty} \leq a$,
%		$$t_{conc}(\eta, \varepsilon) = \mathcal{O}\left(n^2 2^n e^{3 \delta + \frac{3}{\sqrt{\alpha}}\frac{\pi}{2}} \dfrac{a^2(1+\eta)^2}{\eta^2} \ln \left(\dfrac{1}{\varepsilon}\right)\right).$$
		
		\item\label{it:Potts2}[Rapid mixing: mixing time is at most polynomial in $n$ in the modified Glauber dynamics with time-dependent penalty] Assume the logarithmic annealing schedule $\alpha_t = \frac{3 \pi}{2} \dfrac{1}{\ln (e^{\frac{3\pi}{2}}+p t)}$, where $p$ is chosen as in Theorem \ref{thm:Mixingx2}, then
		\textcolor{black}{$$t_{mix}^f(\varepsilon) = \mathcal{O}\left(n^4 e^{3 \delta} \left(\dfrac{1}{\varepsilon}\right)^{\frac{3\pi}{2n^2 \ln q}}\right).$$}
		
		\item\label{it:Potts3}[Torpid mixing: mixing time is at least exponential in $n$ in the original Glauber dynamics]
		Assume in addition that $2 \leq q \leq 4$. For constant $\mathbb{M} = \mathbb{M}(q)$ that does not depend on $n$,
		$$t_{mix}^0(\varepsilon) = \Omega\left(e^{\mathbb{M} n}\ln \left(\dfrac{1}{2 \varepsilon}\right)\right).$$
	\end{enumerate}
\end{theorem}

\begin{rk}
	For $m \in \{1,\ldots,q\}$, we can alternatively take $c = \mcH(\textbf{m}) + \delta = \mcH(\textbf{1}) + \delta$, where we recall that there are $q$ ground-state configurations for the $q$-state Potts model.
\end{rk}

\subsubsection{Proof of Theorem \ref{thm:MixingGlauber_Potts}}

 We first state a spectral gap lower bound that generalizes Lemma \ref{lem:Glauberspectrapgap}:

\begin{lemma}\label{lem:Glauberspectrapgap_Potts}
	Given $G = \left(\mathbb{Z}/ n\mathbb{Z}\right)^2 = (V, E)$, consider the Glauber dynamics on $\mathcal{X} = \{1,2,\ldots,q\}^V$ with proposal generator $P^{Potts} - I$ and target distribution $\modcpi$, then 
	\begin{align*}
		\lambda_2(-\modcM) \geq \dfrac{q}{(q-1)|V|} e^{-3 \max_{\sigma}\modcH(\sigma)}.
	\end{align*}
	In particular, if we take $f(x) = x^2$ and $c = \mcH(\mathbf{1}) + \delta$, where $\delta > 0$, we have
	\begin{align*}
		\lambda_2(-\modcM) \geq \dfrac{q}{(q-1)|V|} e^{- 3 \delta - \frac{3}{\sqrt{\alpha}}\frac{\pi}{2}}.
	\end{align*}
\end{lemma}

\begin{proof}[Proof of Lemma \ref{lem:Glauberspectrapgap_Potts}]
	To prove the desired result, we again invoke Lemma \ref{lem:relationSpeGapL1L2}: the desired result follows if we take $L_1 = \modcM$, $L_2 = P^{Potts}-I$ and $A = \exp\bigg\{-\max_{\sigma}\modcH(\sigma) \bigg\}$ therein,
	and recall that in \cite[Corollary $12.13$]{LPW17} and \cite[Chapter $5$ Example $9$]{AF14} the spectral gap of $P^{Potts}$ can be computed to be $\lambda_2(-(P^{Potts}-I)) = \frac{q}{(q-1)|V|}$. 
	
	In the special case when $f(x) = x^2$ and $c = \mcH(\mathbf{1}) + \delta$, we make use of Section \ref{subsec:quadraticf} and  $\sup_x \arctan(x) = \pi/2$.
	
	%	The above computation leads to
	%	\begin{align*}
		%	\lambda_2(-\modcM) &\geq 2 N^{-1} e^{-3(\max_{\sigma \in \Sigma_N} \mcH^{f_{2}}_{T_t,c_N}(\sigma) - \min_{\sigma \in \Sigma_N} \mcH^{f_{2}}_{T_t,c_N}(\sigma))} \\
		%	&\geq 2 N^{-1} e^{-\beta_t 3 \left(2 \epsilon_N + \frac{\pi}{2}\right)}.
		%	\end{align*}
\end{proof}

We proceed to give the proof of Theorem \ref{thm:MixingGlauber_Potts}.

\begin{proof}[Proof of Theorem \ref{thm:MixingGlauber_Potts}]
	First, we prove item \eqref{it:Potts2}, which simply follows from Theorem \ref{thm:Mixingx2} and the observation from Lemma \ref{lem:Glauberspectrapgap_Potts} that
	$$\mcH_{\max }-\mcH_{\min } = \mathcal{O}(n^2), \quad K = \dfrac{q}{(q-1)4n^2} e^{-3 \delta}.$$
	
	Next, we prove item \eqref{it:Potts3}. Using \cite[Lemma $20.11$]{LPW17}, we see that
	\begin{align*}
		t_{mix}^0(\varepsilon) \geq \dfrac{1}{\lambda_2(-M^0)} \ln \left(\dfrac{1}{2 \varepsilon}\right).
	\end{align*}
	The desired result follows from \cite[Theorem $4$]{GL18} since 
	$$\lambda_2(-M^0) = \mathcal{O}(e^{\mathbb{M} n}).$$
\end{proof}

\section*{Acknowledgements}

The first author Michael Choi acknowledges the financial support from the startup grant and the special pocket research grant from Yale-NUS College, National University of Singapore. The second author Jing Zhang acknowledges the financial support from AIRS - Shenzhen Institute of Artificial Intelligence and Robotics for Society Project 2019-INT002, NSFC Project number 12001460, and from SRIBD - Shenzhen Research Institute of Big Data.

\bibliographystyle{abbrvnat}
\bibliography{thesis}

\end{document}